\documentclass[12pt]{amsart}
\usepackage[T1]{fontenc}
\usepackage[utf8]{inputenc}
\usepackage{lmodern,amsmath,amssymb,amsthm,mathtools,euscript,mathrsfs}
\usepackage[a4paper,hmargin=29mm,vmargin=27mm]{geometry}
\usepackage{microtype}
\IfFormatAtLeastTF{2025-06-01}{\IfPackageAtLeastTF{microtype}{2025-07-09}{}{\microtypesetup{expansion=false}}}{}
\usepackage{enumitem,tikz-cd,booktabs,needspace}
\usepackage[colorlinks=true,allcolors=blue!45!black]{hyperref}
\usepackage{xcolor}
\usepackage[nameinlink,noabbrev]{cleveref}
\setlist[enumerate]{label=\textup{(\roman*)},leftmargin=*,itemsep=3pt}
\numberwithin{equation}{section}
\newtheorem{theorem}{Theorem}[section]
\newtheorem{proposition}[theorem]{Proposition}
\newtheorem{lemma}[theorem]{Lemma}
\newtheorem{corollary}[theorem]{Corollary}
\theoremstyle{definition}
\newtheorem{definition}[theorem]{Definition}
\newtheorem{example}[theorem]{Example}
\theoremstyle{remark}
\newtheorem{remark}[theorem]{Remark}
\newcommand{\Q}{\mathbb Q}
\newcommand{\C}{\EuScript C}
\newcommand{\R}{\Vect_k^{\C}}
\newcommand{\B}{\Vect_{\Q}^{\C^{\op}}}
\newcommand{\X}{\mathscr X}
\newcommand{\M}{\mathcal M}
\newcommand{\A}{\mathcal A}
\newcommand{\kk}{\underline{k}}
\newcommand{\op}{\mathrm{op}}
\newcommand{\PL}{\mathrm{PL}}
\newcommand{\qi}{\xrightarrow{\simeq}}
\newcommand{\sSet}{\mathsf{sSet}}
\newcommand{\DGCA}{\mathsf{DGCA}}
\newcommand{\Ch}{\mathsf{Ch}}
\newcommand{\Vect}{\mathsf{Vect}}
\newcommand{\Ho}{\operatorname{Ho}}
\newcommand{\D}{\mathrm D}
\DeclareMathOperator{\Fun}{Fun}
\DeclareMathOperator{\Hom}{Hom}
\DeclareMathOperator{\RHom}{RHom}
\DeclareMathOperator{\Ext}{Ext}
\DeclareMathOperator{\Aut}{Aut}
\DeclareMathOperator{\Soc}{Soc}
\DeclareMathOperator{\im}{im}
\DeclareMathOperator{\coker}{coker}
\DeclareMathOperator{\id}{id}
\DeclareMathOperator{\Cone}{Cone}
\DeclareMathOperator{\Der}{Der}
\DeclareMathOperator{\gldim}{gldim}
\DeclareMathOperator{\Inj}{Inj}
\DeclareMathOperator{\Proj}{Proj}
\DeclareMathOperator{\Sym}{Sym}
\DeclareMathOperator{\cdim}{cd}
\DeclareMathOperator{\Orb}{Orb}
\DeclareMathOperator{\sd}{sd}
\newcommand{\coind}[2]{(\mathrm{incl}_{#1})_*#2}
\newcommand{\Ic}[1]{\underline{k_{#1}}}
\newcommand{\MinModels}{\mathrm{MinModels}}
\newcommand{\AlgModels}{\mathrm{AlgModels}}
\newcommand{\Spaces}{\Fun(\C^{\op},\sSet_*)_{\Q}^{\mathrm{sc,ft}}}
\newcommand{\AlgCat}{\DGCA_{\Q,\mathrm{fin}}^{\mathrm{sc}}(\C)}
\newcommand{\Qind}{\mathcal Q}
\newcommand{\LQ}{\mathbf L\mathcal Q}
\title[DG-minimal models of diagrams]{Algebraically and geometrically minimal\protect\\DG-models of diagrams of spaces}
\author[Nikita Golub]{Nikita Golub\textsuperscript{\(\dagger\)}}
\thanks{\textsuperscript{\(\dagger\)} Mathematics Program and Center for Quantum and Topological Systems
(CQTS), New York University Abu Dhabi, UAE; Department of Mathematics and Computer Science,
St. Petersburg State University, St. Petersburg, Russia.}
\thanks{\textit{Acknowledgements.} The author gratefully acknowledges financial support
from a scholarship in honour of V.~A.~Rokhlin.}
\date{September 15, 2026}
\hypersetup{pdftitle={Algebraically and geometrically minimal DG-models of diagrams of spaces},pdfauthor={Nikita Golub}}

\begin{document}
\raggedbottom
\begin{abstract}
We extend Sullivan's rational homotopy theory by constructing geometrically and algebraically minimal models for diagrams of simply connected spaces over a class of finite indexing categories. Both classes classify rational diagram homotopy types, but their minimality conditions can differ. We also show that, in contrast to classical rational homotopy theory, H-diagrams need not split into products of Eilenberg--MacLane diagrams, and characterize the indexing categories for which every H-diagram admits such a splitting. Also we apply this to computation of automorphism groups of rational homotopy types of certain diagrams of spaces.
\end{abstract}
\maketitle
\enlargethispage{3pt}

\clearpage
\tableofcontents
\clearpage

\section{Introduction}\label{sec:intro}
A simply connected space $Y$ is \emph{rational} if every $\pi_n(Y)$ is a rational vector space. A \emph{rationalization} of a pointed simply connected space $X$ is a pointed map $\ell_X:X\to X_{\Q}$ to a simply connected rational space such that
\[
 \pi_n(\ell_X)\otimes\Q:\pi_n(X)\otimes\Q
 \xrightarrow{\cong}\pi_n(X_{\Q})\qquad(n\geq2).
\]
Rationalizations exist and are characterized, up to homotopy under $X$, by the universal property
\[
 \ell_X^*:[X_{\Q},Y]_*\xrightarrow{\cong}[X,Y]_*
 \qquad\text{for every simply connected rational space }Y;
\]
here brackets denote pointed homotopy classes. They preserve rational cohomology \cite[Section~1.1]{Hess06}.

Rationalization simplifies many homotopy calculations. For example, $(S^{2r+1})_{\Q}\simeq K(\Q,2r+1)$ for $r\geq1$, where $K(\Q,n)$ is an Eilenberg--MacLane space. An even sphere $S^{2r}$ has just two nonzero rational homotopy groups, in degrees $2r$ and $4r-1$; for $\mathbb{CP}^m$ they occur in degrees $2$ and $2m+1$, with $m\geq1$; see \cite[Section~2.1]{Hess06}. The simplicity of these algebraic descriptions motivates an analogous theory for diagrams of spaces, with models that encode both the spaces and all maps between them.

In Sullivan's formulation, the algebraic counterpart of a space $X$ is the differential graded-commutative algebra (DGCA) of polynomial forms $\A_{\PL}(X)$ \cite{Sullivan77,BG76}. It plays the role of the de Rham complex while retaining rational coefficients and applying to arbitrary simplicial sets.

For a simply connected space of finite rational type, a minimal Sullivan model is a quasi-isomorphism
\[
 \M(X)=(\Sym V,d)\qi\A_{\PL}(X),
 \qquad V=\bigoplus_{n\geq2}V^n,\qquad dV\subset\Sym^{\geq2}V.
\]
Here $\Sym$ denotes the free graded-commutative algebra: generators of even degree are polynomial and those of odd degree are exterior. The differential is determined on generators and extended by the graded Leibniz rule. The model recovers both cohomology and rational homotopy:
\[
 H^*(\M(X))\cong H^*(X;\Q),\qquad
 \mathcal Q(\M(X))^n\cong(\pi_n(X)\otimes\Q)^\vee,
\]
where $\mathcal Q(A)=A_+/(A_+)^2$ is the complex of indecomposables. Since the differential of a minimal Sullivan algebra is decomposable, its indecomposable differential is zero. Minimal Sullivan models exist and are unique up to isomorphism; their homotopy classes of maps describe the rational homotopy category \cite{FHT01}.

The functor of polynomial forms is contravariant:
\[
 \A_{\PL}:(\sSet_*)^{\op}\longrightarrow\DGCA_{\Q}^{\geq0};
\]
evaluation at the basepoint supplies the augmentation. Applied objectwise, it induces
\[
 \A_{\PL}:\Fun(\C^{\op},\sSet_*)^{\op}
 \longrightarrow\Fun(\C,\DGCA_{\Q}^{\geq0}).
\]
Thus an arrow $f:c\to a$ gives a space map $\X(a)\to\X(c)$ and an algebra map $\A_{\PL}(\X(c))\to\A_{\PL}(\X(a))$. Rational homotopy groups form contravariant systems $\underline\pi_n(\X)$, whereas their objectwise duals and the systems of algebra generators are covariant. The polynomial-forms/realization correspondence for diagrams is recalled in \eqref{eq:rational-equivalence}.

Sullivan minimal models are not functorial as strict DGCA models: maps of spaces determine homotopy classes of model maps, without a canonical choice of representatives compatible with composition. To construct minimal models of diagrams, we overcome this obstruction by enlarging the generator systems $\underline\pi_n(\X)^\vee$, the covariant duals of the rational homotopy groups, to their minimal injective resolutions in $\Vect_{\Q}^{\C}$. The additional resolution generators allow the structure maps to commute with the differential and satisfy all composition relations strictly. Thus the enlarged algebras form an actual DGCA diagram representing the original rational homotopy type (\cref{def:elementary,thm:geometric-existence}).

This issue is already central in equivariant rational homotopy theory. For a finite group $G$, the \emph{orbit category} $\Orb_G$ has the transitive $G$-sets $G/H$, for subgroups $H\leq G$, as objects and $G$-equivariant maps as morphisms. A $G$-space $X$ determines the contravariant fixed-point diagram $G/H\mapsto\operatorname{Sing}(X^H)$. The equivariant models of Triantafillou and Rothenberg--Triantafillou use injective systems over this category \cite{Triantafillou82,RothenbergTriantafillou84}. Scull developed minimal models constructed by elementary extensions and distinguished their minimality from the earlier algebraic condition \cite{Scull02,Scull08}. Related constructions over EI-categories appear in Golasi\'nski's work \cite{Golasinski98}.

The existence of algebraically minimal models requires a separate argument, which, to our knowledge, is missing from the literature. Scull's example \cite[Section~21]{Scull02} shows that the elementary-extension construction does not in general produce the decomposability required in \cite[Definition~5.1]{Triantafillou82}. Our cancellation theorem supplies that argument for a stronger notion incorporating cofibrancy (\cref{def:algebraic,thm:algebraic-existence}). It produces an algebraically minimal model as a deformation retract of a geometric model (\cref{lem:algebraic-retract}). This existence result, together with the comparison of the two notions, is a principal contribution of the paper.

Our indexing categories are \emph{good}: they have finitely many morphisms, and every endomorphism is invertible. For the minimal-model constructions we also impose the $\otimes$-condition: the objectwise tensor product of injective systems is injective (\cref{def:good}). All finite orbit categories satisfy this condition by \cref{prop:orbit}, as do finite meet semilattices and subdivisions $\sd(P)$ of finite posets by \cref{ex:meet,ex:subdivision}. The latter index the flag diagrams used in stratified homotopy theory. \Cref{ex:tensorfailure} shows why a hypothesis is needed for general finite posets. We write $\cdim_k(\C)$ for the global dimension of the category algebra $k\C$, the supremum of the injective dimensions of all coefficient systems (\cref{prop:dimension}).

\Needspace{7\baselineskip}
Our main result is the construction of \emph{two classes of minimal DG-models} for diagrams indexed by good categories satisfying the $\otimes$-condition. A \emph{geometrically minimal model} $\M(\X)$ follows the Postnikov tower: at the passage from $\X_n$ to $\X_{n+1}$, it adjoins the minimal injective resolution of $\underline\pi_{n+1}(\X)^\vee$, with differential determined by the $k$-invariant (\cref{thm:geometric-existence,lem:postnikov}). An \emph{algebraically minimal model} removes the contractible summands of the complex of indecomposables, including those joining different Postnikov stages (\cref{lem:cancellation,thm:algebraic-existence}). Its intrinsic minimality condition is
\[
 d_0\bigl(\Soc\mathcal Q(\M)\bigr)=0,
\]
where $d_0$ is the linear differential and the socle is the sum of all simple coefficient subsystems (\cref{def:socle,def:algebraic}).

The existence theorems apply to every cohomologically simply connected DGCA diagram of finite type. The two full categories of minimal models are denoted by $\MinModels_{\Q}(\C)$ and $\AlgModels_{\Q}(\C)$. A quasi-isomorphism of geometrically minimal models is homotopic to an isomorphism (\cref{thm:geometric-rigidity}); a quasi-isomorphism of algebraically minimal models is itself an isomorphism (\cref{thm:algebraic-rigidity}). With morphisms taken modulo homotopy, both categories classify rational diagram homotopy types (\cref{cor:classification}).

\emph{In contrast to classical Sullivan theory, the two constructions can give different (non-isomorphic) minimal models of the same rational homotopy type.} A geometric model can retain a contractible pair of generators belonging to different Postnikov stages, whereas an algebraic model cancels it. We establish an obstruction theory for their coincidence: by \cref{thm:spectral-criterion}, a geometric model is algebraically minimal exactly when the hyper-Ext spectral sequence for its derived indecomposables has no nonzero higher differential for any simple coefficient system. The commutative square in \cref{subsec:square} exhibits a nonzero obstruction and the corresponding cancellation explicitly. \Cref{sec:examples} gives further worked models, including tubular-neighbourhood squares of smooth embeddings and a reflection action on a sphere.

Finally, \cref{thm:h-splitting} characterizes the good categories for which every simply connected H-diagram of finite rational type splits rationally into Eilenberg--MacLane diagrams: the exact condition is $\cdim_{\Q}(\C)\leq1$. Under the $\otimes$-condition, the same bound characterizes coincidence of the two minimality notions for all diagrams (\cref{thm:equivalence}). For finite group actions, it holds precisely for the trivial group and cyclic groups of prime-power order (\cref{cor:groups}).

All spaces are pointed and objectwise simply connected, and have finite rational type. The Postnikov fibres therefore have untwisted coefficients. General nilpotent diagrams require additional fundamental-group actions and local coefficients.

\section{Representations of finite EI-categories}\label{sec:linear}
Fix a field $k$ of characteristic zero. Rational homotopy interpretations use $k=\Q$.
\Needspace{8\baselineskip}
\subsection{Conventions and injective envelopes}
\begin{definition}\label{def:good}
A \emph{good category} is a category with finitely many morphisms in which every endomorphism is invertible. We write
\[
 \R=\Fun(\C,\Vect_k),\qquad G_c=\Aut_{\C}(c).
\]
The additional \emph{$\otimes$-condition} means that $I\otimes_k J$ is injective in $\R$ whenever $I$ and $J$ are injective; tensor products are taken objectwise.
\end{definition}
A \emph{skeleton} is the full subcategory containing one representative of each isomorphism class of objects. Its inclusion into $\C$ is an equivalence, so restriction induces an equivalence of functor categories, compatible with objectwise tensor products and preserving injective objects. We henceforth replace $\C$ by a skeleton and keep the same notation. This makes distinct objects nonisomorphic and avoids counting the same coefficient data more than once in the formulas below.

Nonisomorphisms in a good category cannot form a cycle. Hence there is a function $[-]:\operatorname{Ob}\C\to\{0,\ldots,L\}$ strictly decreasing along nonisomorphisms. Thus $[a]>[b]$ for a nonisomorphism $a\to b$. There need not be a terminal object. Each $G_c$ is finite. The \emph{category algebra} $k\C$ is the finite-dimensional vector space with the morphisms of $\C$ as basis; multiplication is composition when defined and zero otherwise. We identify a functor $F:\C\to\Vect_k$ with the left $k\C$-module $\bigoplus_c F(c)$, on which $f:c\to a$ acts by $F(f)$ on $F(c)$ and by zero on the other summands. Conversely, a module $M$ gives $F(c)=1_cM$, where $1_c$ is the identity morphism of $c$.

A \emph{system of vector spaces} is an object $F\in\Vect_k^{\C}$. A subsystem $E\subset F$ consists of subspaces $E(c)\subset F(c)$ satisfying
\[
 F(f)(E(c))\subset E(a)\qquad\text{for every }f:c\to a.
\]
The maps of $E$ are the restrictions of those of $F$. For an arrow $u:F(c)\to F(a)$, this requirement is $u(E(c))\subset E(a)$; it is the compatibility condition used in forming the socle below. A nonzero system is \emph{simple} if it has no proper nonzero subsystem. For any $kG_c$-module $U$, define the supported system
\[
 S_c(U)(a)=\begin{cases}U,&a=c,\\0,&a\ne c.\end{cases}
\]
Automorphisms of $c$ act on $U$ as prescribed, and all nonisomorphisms act by zero. The system $S_c(U)$ is simple exactly when $U$ is a simple $kG_c$-module, and every simple system has this form. Indeed, let $J\subset k\C$ be the ideal spanned by nonisomorphisms. It is nilpotent, and
\[
 k\C/J\cong\prod_{c}kG_c
\]
is semisimple. Hence $J$ is the radical of $k\C$, and every simple module is annihilated by $J$.

\begin{definition}[Essential extension]\label{def:essential}
An inclusion $E\subset F$ is \emph{essential} if $N\cap E\ne0$ for every nonzero subsystem $N\subset F$. If $F$ is also injective, the inclusion is an \emph{injective envelope} (or \emph{injective hull}) of $E$; see \cite[Tag~08Y1]{Stacks}.
\end{definition}

\begin{definition}[Socle]\label{def:socle}
The \emph{socle} $\Soc F$ is the sum, inside $F$, of all its simple subfunctors. At $c$ define
\begin{equation}\label{eq:socle}
 F_c=\bigcap_{\substack{f:c\to a\\f\text{ nonisomorphism}}}\ker\bigl(F(f):F(c)\to F(a)\bigr).
\end{equation}
The intersection is $F(c)$ if there are no such morphisms. Then
\[
 (\Soc F)(c)=F_c,\qquad
 \Soc F=\bigoplus_c S_c(F_c).
\]
The space $F_c$ is stable under $G_c$ and is a direct sum of simple $kG_c$-modules, since $kG_c$ is semisimple. Thus $S_c(F_c)$ collects all simple subsystems supported at $c$; it need not itself be simple.
\end{definition}

Indeed, the annihilator of $J$ is precisely the system in \eqref{eq:socle}; it is a module over the semisimple quotient $k\C/J$, hence a sum of simples. In particular, every nonisomorphism acts by zero on the socle. Every nonzero subsystem $N\subset F$ meets $\Soc F$ nontrivially: choose the largest $r$ with $J^rN\ne0$; then $0\ne J^rN\subset N\cap\Soc F$. Thus $\Soc F\subset F$ is essential in the sense of \cref{def:essential}. Consequently, a system with zero socle is zero.

For example, on the arrow $c\to a$ a system is a linear map $u:F(c)\to F(a)$, and
\[
 \Soc F=\bigl(\ker u\xrightarrow{\,0\,}F(a)\bigr).
\]
For the identity system $(k\xrightarrow{\id}k)$ the socle is $(0\to k)$, not the entire system. The definition is applied degree by degree to a graded system: $(F^n)_c$ denotes the socle component in degree $n$.

To construct injective systems, let $\mathrm{incl}_c:G_c\to\C$ denote the inclusion of the one-object automorphism category. Restriction to $c$ and right Kan extension give the adjunction
\[
\begin{tikzcd}[column sep=large]
 \Vect_k^{\C}\arrow[r,shift left=1.5,"\mathrm{ev}_c"]
 &kG_c\text{-}\mathrm{Mod}\arrow[l,shift left=1.5,"(\mathrm{incl}_c)_*"] .
\end{tikzcd}
\]
For a $kG_c$-module $U$, the right adjoint is
\begin{equation}\label{eq:coinduced}
 (\coind{c}{U})(a)=\Hom_{kG_c}(k\C(a,c),U).
\end{equation}
For $g:a\to b$, its structure map sends $\phi$ to the function $f\mapsto\phi(fg)$. The adjunction reads
\[
 \Hom_{\Vect_k^{\C}}(F,\coind{c}{U})
 \cong\Hom_{kG_c}(F(c),U).
\]
Evaluation is exact and $kG_c$ is semisimple, so every system \eqref{eq:coinduced} is injective. Its socle is $S_c(U)$: a socle element at $a\ne c$ is killed by evaluation along every $a\to c$, and must vanish.

Following the notation of the injective-envelope construction, write
\[
 \underline{F_c}:=\coind{c}{F_c}.
\]
The three objects $F_c$, $S_c(F_c)$, and $\underline{F_c}$ are distinct: they are respectively a $kG_c$-module, the corresponding socle system supported at $c$, and its injective envelope, which can be nonzero at objects mapping to $c$.

\begin{proposition}\label{prop:envelope}
Every system $F\in\Vect_k^{\C}$ has an injective envelope
\begin{equation}\label{eq:envelope}
 F\lhook\joinrel\longrightarrow\Inj(F):=\bigoplus_c\underline{F_c}.
\end{equation}
Every injective system is a direct sum of systems \eqref{eq:coinduced}. Finite-dimensional systems have finite-dimensional injective envelopes.
\end{proposition}
\begin{proof}
Choose $G_c$-equivariant retractions $p_c:F(c)\to F_c$. The adjunction defines the natural map
\[
 F(a)\longrightarrow\bigoplus_c\Hom_{kG_c}(k\C(a,c),F_c),
 \qquad x\longmapsto\bigl[f\mapsto p_cF(f)x\bigr]_c.
\]
On the socle it is the identity under the preceding identification. Its kernel therefore has zero socle, and is zero. Its image contains the entire socle of its injective target, so the inclusion is essential. If $F$ is injective, the inclusion splits; the complementary summand has zero socle and hence vanishes. Finiteness follows directly from the formula.
\end{proof}

The retractions $p_c$ are choices; the envelope is unique up to an isomorphism over $F$ \cite[Tag~08Y1]{Stacks}. Starting with $F^{(0)}=F$, define successively
\[
 \Inj^j(F)=\Inj(F^{(j)}),\qquad
 F^{(j+1)}=\Inj^j(F)/F^{(j)}.
\]
The quotient map followed by the next envelope inclusion gives the \emph{minimal injective resolution}
\begin{equation}\label{eq:resolution}
\begin{tikzcd}[column sep=small]
 0\arrow[r]&F\arrow[r,hook]&\Inj(F)\arrow[r,"\partial"]
 &\Inj^1(F)\arrow[r,"\partial"]&\Inj^2(F)\arrow[r]&\cdots .
\end{tikzcd}
\end{equation}
We also write $\Inj^0(F)=\Inj(F)$. Each differential vanishes on the socle, since its kernel is the essential subobject $F^{(j)}$. Conversely, a resolution by injectives whose differential vanishes on every socle is minimal: its cycles contain the socle, and hence are essential. These resolutions are unique up to an isomorphism of complexes extending the identity of $F$: comparison maps between injective resolutions are homotopy equivalences, and \cref{lem:mincomplex}(ii) makes them isomorphisms.

On the arrow $c\to a$, for instance, the minimal resolution of the system supported at $a$ is
\[
 0\longrightarrow(0\to k)\longrightarrow(k\xrightarrow{\id}k)
 \longrightarrow(k\to0)\longrightarrow0.
\]
Its differential is nonzero at $c$ but zero on the socle, which is supported at $a$. Thus minimality of an injective complex does not mean that its differential vanishes objectwise.

For comparison, the dual construction takes place in $\Vect_k^{\C^{\op}}$. For a contravariant system $P$, put
\[
 P_c=P(c)\Big/\sum_{f:c\to a\text{ nonisomorphism}}\im P(f).
\]
It is a right $kG_c$-module. Choosing equivariant sections gives the projective cover
\[
 \Proj(P)=\bigoplus_c P_c\otimes_{kG_c}k\C(-,c)
 \twoheadrightarrow P,
 \qquad v\otimes f\longmapsto P(f)v.
\]
Its kernel is superfluous: any submodule mapping onto $P$ must equal $\Proj(P)$, by nilpotence of the radical. For finite-dimensional systems, objectwise duality exchanges this cover with \eqref{eq:envelope} and exchanges projective and injective resolutions.

For a system $F$, its \emph{injective dimension} $\operatorname{id}_{\R}F$ is the least $d\geq0$ for which there is an exact sequence $0\to F\to I^0\to\cdots\to I^d\to0$ with every $I^j$ injective; it is $\infty$ if no such $d$ exists.

\begin{proposition}\label{prop:dimension}
The number
\[
 \cdim_k(\C)=\gldim(k\C)=\sup_{F\in\R}\operatorname{id}_{\R}F
\]
is finite and at most $L$, where $L$ is the maximal length of a chain of nonisomorphisms. The global dimensions of $k\C$ and $k\C^{\op}$ agree.
\end{proposition}
\begin{proof}
We construct a projective resolution of length at most $L$ for every left module $F$. Recall the normalized relative bar complex for a split algebra $R=D\oplus J$, where $D$ is a subalgebra and $J$ an ideal. Its terms and augmentation are
\[
 B_r(R,D;F)=R\otimes_D J^{\otimes_D r}\otimes_D F\quad(r\geq0),
 \qquad J^{\otimes_D0}=D.
\]
Write $\otimes$ for $\otimes_D$ in the following formulas. The augmentation is $b_0:B_0\to F$, $b_0(a\otimes v)=av$. For $r\geq1$, the differential is
\begin{align*}
 &b_r(a\otimes j_1\otimes\cdots\otimes j_r\otimes v)\\
 &\quad=aj_1\otimes j_2\otimes\cdots\otimes j_r\otimes v\\
 &\qquad+\sum_{i=1}^{r-1}(-1)^i
 a\otimes j_1\otimes\cdots\otimes(j_ij_{i+1})\otimes\cdots\otimes j_r\otimes v\\
 &\qquad+(-1)^r a\otimes j_1\otimes\cdots\otimes j_{r-1}\otimes(j_rv).
\end{align*}
Adjacent multiplications cancel in pairs, giving $b_{r-1}b_r=0$. If $\bar a$ is the $J$-component of $a$, a $D$-linear contraction of the augmented complex is
\begin{align*}
 s_{-1}(v)&=1\otimes v,\\
 s_r(a\otimes j_1\otimes\cdots\otimes j_r\otimes v)
 &=1\otimes\bar a\otimes j_1\otimes\cdots\otimes j_r\otimes v
 \qquad(r\geq0).
\end{align*}
For $r=0$ the second formula reads $s_0(a\otimes v)=1\otimes\bar a\otimes v$. In $b_{r+1}s_r+s_{r-1}b_r$, the internal multiplication terms cancel. Balancing over $D$ identifies the remaining terms with
\[
 \bar a\otimes j_1\otimes\cdots\otimes j_r\otimes v
 +(a-\bar a)\otimes j_1\otimes\cdots\otimes j_r\otimes v,
\]
whose sum is the original tensor; the same calculation applies at $r=0$. Also $b_0s_{-1}=\id_F$. Thus the complex is exact, although the contraction need not be $R$-linear.

Apply this construction to $R=k\C$, $D=\prod_c kG_c$, and the ideal $J$ of nonisomorphisms. Each $B_r$ is induced from a $D$-module and is projective because $D$ is semisimple. For $r>L$ it vanishes: a nonzero balanced tensor of morphisms requires a composable chain of $r$ nonisomorphisms. Thus every module has projective dimension at most $L$. The same argument applies to the opposite category. For a finite-dimensional algebra, global dimension is detected by Ext between finite-dimensional modules; objectwise duality identifies these Ext groups with those over the opposite algebra. The two global dimensions therefore agree, as do the suprema of projective and injective dimensions.
\end{proof}

The notation $\cdim_k(\C)$ denotes the \emph{global dimension} of the category algebra, to distinguish it from the cohomological dimension of the constant representation, which may be smaller. In particular, the square in \cref{subsec:square} has a terminal object but satisfies $\cdim_k(\C)=2$.

\subsection{Tensor products of injective systems}
\begin{lemma}\label{lem:tensor}
Assume the $\otimes$-condition. If $\underline V=\bigoplus_{n\geq1}\underline V^n$ is degreewise injective, then every positive-degree component of the free graded-commutative algebra $\Sym \underline V$ is injective. The same holds for a tensor product of such algebras. No injectivity assertion about the constant degree-zero representation is needed.
\end{lemma}
\begin{proof}
Each homogeneous component is a direct sum of direct summands of finite tensor products of the $\underline V^n$: use the signed averaging idempotent for symmetric powers. Tensor products are injective by hypothesis. Direct sums of injectives are injective over the Noetherian ring $k\C$.
\end{proof}

\begin{example}[Meet semilattices]\label{ex:meet}
For a finite poset, write $\Ic{c}=\coind{c}{k}$. Explicitly,
\[
 \Ic{c}(a)=\begin{cases}k,&a\leq c,\\0,&a\nleq c,\end{cases}
 \qquad
 \Ic{c}(a\leq b)=\begin{cases}\id_k,&b\leq c,\\0,&b\nleq c.\end{cases}
\]
If every nonempty intersection of two principal down-sets is principal, then
\[
 \Ic{c}\otimes \Ic{d}=\Ic{c\wedge d}
\]
when the intersection is nonempty, and is zero otherwise. By \eqref{eq:envelope}, every injective on a finite poset is a direct sum of copies of the $\Ic{c}$. Tensor products distribute over these sums, and direct sums of injectives are injective. Such posets therefore satisfy the $\otimes$-condition; in particular, so do finite meet semilattices and the commutative square.
\end{example}

\begin{example}[Subdivisions and stratified homotopy theory]\label{ex:subdivision}
Let $P$ be a finite nonempty poset. Its \emph{subdivision} $\sd(P)$ is the poset of nonempty strict chains in $P$, ordered by inclusion. It is finite and has only identity endomorphisms, hence is good. The injectives of \cref{ex:meet} satisfy
\[
 \Ic{\sigma}\otimes_k\Ic{\tau}\cong
 \begin{cases}
 \Ic{\sigma\cap\tau},&\sigma\cap\tau\ne\varnothing,\\
 0,&\sigma\cap\tau=\varnothing.
 \end{cases}
\]
Indeed, a nonempty intersection of chains is a chain; its nonempty subchains form the common support, with identity structure maps. Thus \cref{ex:meet} proves the $\otimes$-condition. The empty chain need not be adjoined.

For $P=\{0<1\}$, the category and a contravariant space diagram on it are respectively
\[
 \sd(P)=\bigl(\{0\}\longrightarrow\{0<1\}\longleftarrow\{1\}\bigr),
 \qquad X_0\longleftarrow X_{01}\longrightarrow X_1.
\]
In stratified homotopy theory, this span records the strata and their homotopy link. More generally, the d\'ecollage description of abstract $P$-stratified homotopy types uses coherent contravariant space diagrams on $\sd(P)$ satisfying the Segal condition: each flag space is the iterated homotopy pullback of its consecutive link spaces over the intermediate strata \cite[Definition~1.1.4 and Theorem~1.1.7]{Haine23}. Polynomial forms give covariant DGCA diagrams on $\sd(P)$. With simply connected values of finite rational type and coherent basepoints, both minimal-model constructions apply. Thus our current theory opens up a possibility for the DG-rational theory for stratified homotopy theory. 
\end{example}

\Needspace{7\baselineskip}
\begin{example}[A necessary restriction]\label{ex:tensorfailure}
The $\otimes$-condition does not hold for every finite poset. Consider the order generated by
\[
 r<a,b<c,d,
\]
where $a,b$ are incomparable, as are $c,d$. Then $\Ic{c}\otimes \Ic{d}$ is supported on $\{r,a,b\}$, is one-dimensional at each of these objects, and has identity structure maps. Its socle has one copy of $k$ at $a$ and at $b$. Its injective envelope is $\Ic{a}\oplus \Ic{b}$, whose value at $r$ has dimension two. The tensor product is therefore not injective.
\end{example}

For a finite group $G$, the orbit category $\Orb_G$ has objects $G/H$ and equivariant maps as morphisms. Explicitly,
\begin{align*}
 \Orb_G(G/H,G/K)&=\{aK\in G/K\mid a^{-1}Ha\subset K\},\\
 \Aut(G/H)&\cong W_G(H)=N_G(H)/H.
\end{align*}
We may choose one subgroup in each conjugacy class to obtain a skeleton. A pointed $G$-space determines the contravariant fixed-point diagram $\underline X(G/H)=\operatorname{Sing}(X^H)$; this is the motivating example of $\Fun(\C^{\op},\sSet_*)$.

\begin{proposition}[Finite orbit categories]\label{prop:orbit}
The orbit category $\Orb_G$ of every finite group $G$ is good and satisfies the $\otimes$-condition.
\end{proposition}
\begin{proof}
Endomorphisms of a finite transitive $G$-set are invertible, and a skeleton of $\Orb_G$ is finite. For contravariant representations, the projectives are summands of sums of the linear representables $k\Orb_G(-,G/H)$. Their tensor products are projective because
\begin{align*}
 k\Orb_G(-,G/H)\otimes k\Orb_G(-,G/K)
 &\cong k\Hom_G(-,G/H\times G/K),\\
 G/H\times G/K&\cong\coprod_{[g]\in H\backslash G/K}G/(H\cap gKg^{-1}).
\end{align*}
Duality proves the assertion for finite-dimensional injectives. Every injective is a direct sum of finite-dimensional indecomposable injectives, so distributivity of tensor products and the Noetherian property give the general case.
\end{proof}

\subsection{Minimal complexes of injectives}
All complexes are cochain complexes: their differentials have degree $+1$. A complex is \emph{bounded below} if it is zero in every sufficiently negative degree, and \emph{acyclic} if its cohomology vanishes in every degree. Write $Z^nI=\ker(d^n:I^n\to I^{n+1})$ for its system of cycles.

\begin{definition}\label{def:mincomplex}
A bounded-below complex $I$ of injective systems is \emph{minimal} if
\[
 d^n(\Soc I^n)=0\qquad\text{for every }n.
\]
In other words, each simple subsystem of $I^n$ consists of cycles. Equivalently, $Z^nI\subset I^n$ is an injective envelope in every degree: the inclusion is essential exactly when $Z^nI$ contains $\Soc I^n$.
\end{definition}

The word \emph{minimal} expresses the absence of a removable contractible summand. A two-term complex $U\xrightarrow{\id}U$ has zero cohomology but a nonzero socle differential whenever $U\ne0$. \Cref{lem:mincomplex}(i) shows that these are exactly the summands that must be removed to obtain a minimal complex. For vector spaces alone every space equals its socle, so minimality says $d=0$; for systems it is the weaker condition in \cref{def:mincomplex}.

We first recall the homotopical property of injective complexes used in the proof. For complexes $K,I$, the Hom complex has
\[
 \Hom_{\R}^{r}(K,I)=\prod_j\Hom_{\R}(K^j,I^{j+r}),
 \qquad \delta\phi=d_I\phi-(-1)^r\phi d_K.
\]
Its degree-zero cycles are cochain maps. Two such maps $f,g$ are \emph{chain homotopic} if $f-g=d_Ih+hd_K$ for a degree-$-1$ map $h$; their homotopy classes form $H^0\Hom_{\R}^{\bullet}(K,I)$. A \emph{chain homotopy equivalence} is a cochain map admitting an inverse up to this homotopy.

\begin{definition}[K-injectivity]\label{def:kinjective}
Following the standard definition \cite[Definition~13.31.1, Tag~070H]{Stacks}, a complex $I$ is \emph{K-injective} if $\Hom_{\R}^{\bullet}(K,I)$ is acyclic for every acyclic complex $K$. Equivalently, every cochain map from an acyclic complex to $I$ or any of its shifts is chain homotopic to zero.
\end{definition}

\begin{lemma}\label{lem:kinjective}
Every bounded-below complex of injectives is K-injective. A quasi-isomorphism between such complexes is a chain homotopy equivalence.
\end{lemma}
\begin{proof}
We give the null-homotopy construction explicitly; see also \cite[Tag~070G]{Stacks}. Let $f:K\to I$ be a cochain map, with $K$ acyclic and $I^j=0$ for $j<N$. Set $h^j=0$ for $j\leq N$. Suppose that $h$ has been chosen through degree $n$ and that
\[
 f^{n-1}=d_I^{n-2}h^{n-1}+h^n d_K^{n-1}.
\]
The residual map $g^n=f^n-d_I^{n-1}h^n:K^n\to I^n$ satisfies
\[
 g^n d_K^{n-1}
 =d_I^{n-1}\bigl(f^{n-1}-h^n d_K^{n-1}\bigr)=0.
\]
Since $K$ is acyclic, $\ker d_K^n=\im d_K^{n-1}$. Hence $d_K^nx\mapsto g^n(x)$ defines a map $\im d_K^n\to I^n$. Injectivity of $I^n$ extends it to $h^{n+1}:K^{n+1}\to I^n$, giving $f^n=d_I^{n-1}h^n+h^{n+1}d_K^n$. Induction constructs the required homotopy. The same argument applies to every shift of $I$, proving K-injectivity.

\Needspace{9\baselineskip}
For a quasi-isomorphism $f:I\to J$, the \emph{mapping cone}
\[
 \Cone(f)^n=J^n\oplus I^{n+1},\qquad
 d(y,x)=(d_Jy+f(x),-d_Ix)
\]
is acyclic, by its long exact cohomology sequence. K-injectivity therefore implies that precomposition with $f$ is a quasi-isomorphism on Hom complexes with target $I$ or $J$: the cone of each induced map is a shift of the Hom complex from $\Cone(f)$. With target $I$, surjectivity on $H^0$ gives $g:J\to I$ with $gf\simeq\id_I$. With target $J$, injectivity on $H^0$ gives $fg\simeq\id_J$, since $(fg-\id_J)f\simeq0$.
\end{proof}

\Needspace{10\baselineskip}
\begin{lemma}\label{lem:mincomplex}
Let $I$ be a bounded-below complex of injective systems.
\begin{enumerate}
\item There is a decomposition of complexes $I=\underline W\oplus C$, where $\underline W$ is minimal and $C$ is a direct sum of two-term complexes $U\xrightarrow{\id}U$ with $U$ injective.
\item A quasi-isomorphism between bounded-below minimal complexes of injectives is an isomorphism.
\end{enumerate}
\end{lemma}
\begin{proof}
For (i), suppose the residual complex is already minimal below degree $n$. A module map sends a simple submodule either to zero or to a simple submodule, so $d^n$ maps $\Soc I^n$ into $\Soc I^{n+1}$. Choose a complement $T$ to its kernel in the semisimple system $\Soc I^n$. Extend $T\hookrightarrow I^n$ to a map $U=\Inj(T)\to I^n$, using injectivity of $I^n$. Its kernel has zero intersection with the essential subobject $T$, so it is zero. We regard $U$ as a submodule of $I^n$. Likewise, $d^n|_U$ is injective: a nonzero kernel would meet $T$, on which $d^n$ is injective.

Here is the splitting of the whole complex. Since $d^nU\cong U$ is injective, choose a retraction $p:I^{n+1}\to d^nU$ and put
\[
 A=\ker(pd^n),\qquad B=\ker p.
\]
The map $pd^n$ restricts to the isomorphism $d^n:U\to d^nU$, so
\[
 I^n=U\oplus A,\qquad I^{n+1}=d^nU\oplus B,
 \qquad d^n=\begin{pmatrix}d^n|_U&0\\0&d^n|_A\end{pmatrix}.
\]
Moreover, $d^{n-1}(I^{n-1})\subset A$ and $d^{n+1}(d^nU)=0$, by $d^2=0$. Thus $U\xrightarrow{d^n}d^nU$ is a direct summand \emph{as a complex}, and identifying its two terms gives the disk $U\xrightarrow{\id}U$. The remaining terms $A,B$ are injective, being direct summands of injectives.

The residual differential vanishes on $\Soc A$. Indeed, $\Soc A\subset\Soc I^n$ and $d^n(\Soc I^n)=d^nT\subset d^nU$, while $d^n(A)\subset B$; the intersection $d^nU\cap B$ is zero. Starting in the lowest degree and repeating this construction proves (i). There is one step per degree, and each degree is altered only at the two adjacent steps, so the decomposition is well defined even without a finite-dimensionality assumption.

For (ii), let $f:I\to J$ be a quasi-isomorphism of minimal complexes. By \cref{lem:kinjective}, choose a homotopy inverse $g:J\to I$. On $\Soc I$, both terms of $d_Ih+hd_I$ vanish: $d_I$ is zero there, and $h$ sends socles to socles. Hence $gf\simeq\id_I$ restricts to the equality $gf=\id$ on $\Soc I$. Similarly $fg=\id$ on $\Soc J$. Thus $f^n:\Soc I^n\to\Soc J^n$ is an isomorphism for every $n$.

Now $\ker f^n\cap\Soc I^n=0$. Essentiality of $\Soc I^n\subset I^n$ forces $\ker f^n=0$, so $f^n$ is injective. Its image is isomorphic to the injective module $I^n$; therefore $J^n=f^n(I^n)\oplus E$ for some $E$. Surjectivity on socles gives $\Soc J^n\subset f^n(I^n)$, whence $\Soc E=0$. A nonzero system has nonzero socle, so $E=0$. Every $f^n$ is therefore an isomorphism, as required.
\end{proof}

\section{Homotopy theory and elementary extensions}\label{sec:homotopy}
\subsection{The algebraic setting}
We use DGCAs augmented over $k$ and write $A_+=\ker(A\to\kk)$ for the augmentation ideal of a diagram, where $\kk$ is the constant system with value $k$. Morphisms preserve augmentations. We write
\[
 \DGCA_k(\C)=\Fun(\C,\DGCA_k),\qquad
 \DGCA_k^{\geq0}(\C)=\Fun(\C,\DGCA_k^{\geq0}).
\]
The superscript $\geq0$ refers to the underlying grading; differentials have degree $+1$. Weak equivalences of diagrams are objectwise quasi-isomorphisms. A diagram is \emph{cohomologically simply connected} if $H^0(A)=\kk$ and $H^1(A)=0$. Finite type means finite-dimensional cohomology in each degree at every object. All the models constructed below satisfy the stronger conditions
\begin{equation}\label{eq:reduced}
 \M^0=\kk,\qquad \M^1=0,\qquad \M^n=0\quad(n<0).
\end{equation}

When a free diagram is written $\M=\Sym\underline V$, its value on an arrow $f:c\to a$ is the DGCA diagram
\[
\begin{tikzcd}[column sep=large]
 (\Sym\underline V(c),d_c)
 \arrow[r,"\Sym\underline V(f)"]
 &(\Sym\underline V(a),d_a).
\end{tikzcd}
\]
The arrow sends a generator $v$ to $\underline V(f)(v)$ and commutes with the displayed differentials. This convention applies to both classes of minimal systems below.

For clarity about fibrancy, we use the injective model structure on unbounded complexes in $\R$, and transfer it to augmented DGCAs along
\begin{equation}\label{eq:transfer-adjunction}
\begin{tikzcd}[column sep=large]
 \Ch(\Vect_k^{\C})_{\mathrm{inj}}
 \arrow[r,shift left=1.5,"\Sym"]
 &\DGCA_k(\C)_{\mathrm{trinj}}
 \arrow[l,shift left=1.5,"\mathcal U"] .
\end{tikzcd}
\end{equation}
Here $\Sym(\underline V)=\kk\oplus\Sym^{\geq1}\underline V$ has its canonical augmentation, and $\mathcal U(A)=A_+$.
The use of unbounded complexes here is only an ambient model-category convention; all relevant objects have bounded-below augmentation ideals.

\begin{proposition}\label{prop:model}
The transferred model structure exists. Its weak equivalences are objectwise quasi-isomorphisms, and its fibrations are the surjections whose augmentation-ideal kernels are K-injective complexes of injectives. In particular, a bounded-below augmented diagram is fibrant exactly when its augmentation ideal is degreewise injective. Cofibrations have the left lifting property against trivial fibrations.
\end{proposition}
\begin{proof}
The injective model structure on $\Ch(\R)$ is combinatorial, with monomorphisms as cofibrations. Let $J$ be a set of generating trivial cofibrations. Evaluation sends every map of $J$ to a trivial cofibration of complexes over $k$. In characteristic zero, the free commutative algebra functor sends such a map to a trivial cofibration of ordinary DGCAs. Evaluation preserves pushouts and transfinite compositions of augmented algebras. Consequently every relative $\Sym(J)$-cell map is an objectwise quasi-isomorphism. The small object argument gives the transfer. The description of fibrations follows from the injective structure on complexes; the bounded-below assertion uses \cref{lem:kinjective}.
\end{proof}

\begin{lemma}\label{lem:connective-replacement}
Every nonnegative augmented DGCA diagram admits a nonnegative fibrant replacement in \cref{prop:model}.
\end{lemma}
\begin{proof}
Put $R=k\C$ and let $D^nU$ be the complex $U\xrightarrow{\id}U$ in degrees $n,n+1$. Apply the small object argument to the maps
\[
 \Sym(D^nL)\longrightarrow \Sym(D^nR),\qquad L\subset R\text{ a left ideal},\quad n\geq0,
\]
factoring $A\to\kk$. All objects adjoined are nonnegative. Each displayed map is an objectwise trivial cofibration: after evaluation, split $L\subset R$ as vector spaces, so the map adjoins a free contractible disk. The resulting map $A\to A^{\mathrm f}$ is consequently a quasi-isomorphism. The lifting property says that every map $L\to(A^{\mathrm f}_+)^n$ extends to $R$. Baer's criterion makes every $(A^{\mathrm f}_+)^n$ injective. The resulting nonnegative augmentation ideal is K-injective, so $A^{\mathrm f}$ is fibrant.
\end{proof}

For a fibrant $B$ satisfying \eqref{eq:reduced}, a based polynomial path object is
\[
 P B=\kk\oplus\bigl(B_+\otimes k[t,dt]\bigr),\qquad |t|=0,\quad d(t)=dt.
\]
The constant-path inclusion is a quasi-isomorphism, and endpoint evaluation $PB\to B\times_{\kk}B$ is a fibration. Its kernel is degreewise a direct sum of copies of the injective components of $B_+$. Maps out of cofibrant objects into fibrant objects, modulo this homotopy relation, compute the localized category.

The indecomposable complex is
\begin{equation}\label{eq:indecomposables}
 \Qind A=A_+/(A_+)^2.
\end{equation}
Its right adjoint sends a complex to its square-zero augmented algebra. That adjoint preserves fibrations and trivial fibrations, so $\Qind$ is left Quillen. In particular, a quasi-isomorphism between cofibrant algebras induces a quasi-isomorphism of indecomposable complexes. We denote its derived functor by $\LQ$.

\subsection{Polynomial forms and diagrams of spaces}
For a pointed simplicial set $X$, evaluation at the basepoint augments $\A_{\PL}(X)$. Polynomial forms on a simplex are
\[
 \A_{\PL}(\Delta^m)=
 k[t_0,\ldots,t_m,dt_0,\ldots,dt_m]/\left(\sum t_i-1,\sum dt_i\right).
\]
The generators satisfy $|t_i|=0$, $|dt_i|=1$, and $d(t_i)=dt_i$. A simplicial operator $\theta:[r]\to[m]$ induces $t_i\mapsto\sum_{\theta(j)=i}t_j$ and the corresponding map on differentials. For an arbitrary simplicial set, compatible polynomial forms on all its simplices define $\A_{\PL}(X)$.

The realization functor is
\[
 \mathcal K(A)_m=\Hom_{\DGCA_k^{\mathrm{unaug}}}(A,\A_{\PL}(\Delta^m)),
\]
where the superscript $\mathrm{unaug}$ specifies that the Hom is taken in unaugmented algebras. The augmentation of $A$ specifies a distinguished constant simplex, making $\mathcal K(A)$ pointed. Maps of augmented algebras induce pointed maps of realizations. The resulting contravariant adjunction is characterized by
\[
 \Hom_{\sSet_*}(X,\mathcal K(A))
 \cong\Hom_{\DGCA_k^{\mathrm{aug}}}(A,\A_{\PL}(X)).
\]
Thus a contravariant diagram $\X:\C^{\op}\to\sSet_*$ gives a covariant augmented DGCA diagram. We write $\underline\pi_n(\X)$ for its rational homotopy system, with value $\pi_n(\X(c))\otimes\Q$ at $c$. Here $\AlgCat$ denotes the full subcategory of $\DGCA_{\Q}^{\geq0}(\C)$ consisting of cohomologically simply connected diagrams with finite-dimensional cohomology in every degree at each object. The notation $\mathrm{sc,ft}$ on the space side has the corresponding objectwise meaning.

\Needspace{9\baselineskip}
\begin{lemma}[Diagrammatic Sullivan equivalence]\label{lem:diagrammatic-sullivan}
The polynomial-forms/realization adjunction induces an equivalence
\begin{equation}\label{eq:rational-equivalence}
\begin{tikzcd}[column sep=large]
 \Ho(\Spaces)\arrow[r,shift left=1.5,"\mathbb D\A_{\PL}"]
 &\Ho(\AlgCat)^{\op}\arrow[l,shift left=1.5,"\mathbb D\mathcal K"] .
\end{tikzcd}
\end{equation}
Both homotopy categories mean localization at the indicated objectwise weak equivalences. This assertion does not require the $\otimes$-condition.
\end{lemma}
\begin{proof}
We use the general principle that a derived adjunction whose unit and counit are equivalences on specified homotopy types remains an equivalence on diagrams of those types. We give the argument here to keep track of replacements and of the opposite category.

Let $\mathcal D=\DGCA_{\Q}^{\geq0}$ be the ordinary model category of augmented nonnegative DGCAs. Equip $\Fun(\C^{\op},\sSet_*)$ with the projective model structure and $\Fun(\C,\mathcal D)$ with the injective diagram model structure; these exist by \cite[Proposition~A.2.8.2]{Lurie09}. Taking the opposite of the latter gives the projective model structure on $\Fun(\C^{\op},\mathcal D^{\op})$. This injective diagram structure on algebras is used only in the present proof and is distinct from the transferred structure of \cref{prop:model}.

The ordinary Sullivan adjunction
$\A_{\PL}:\sSet_*\rightleftarrows\mathcal D^{\op}:\mathcal K$
induces a Quillen adjunction between these projective diagram categories. Indeed, fibrations and trivial fibrations in a projective diagram structure are objectwise, and the right adjoint $\mathcal K$ preserves them objectwise. A projectively cofibrant diagram is objectwise cofibrant, and a projectively fibrant diagram is objectwise fibrant. Consequently, diagram replacements used to form the derived unit and counit also compute the ordinary derived unit and counit after evaluation at any object.

More explicitly, for a projectively cofibrant space diagram $\X$, choose a fibrant replacement $\widehat{\A_{\PL}\X}$ of $\A_{\PL}\X$ in $\Fun(\C^{\op},\mathcal D^{\op})$. The derived unit is represented by the natural map
\[
 \X\longrightarrow\mathcal K\bigl(\widehat{\A_{\PL}\X}\bigr).
\]
At each object this is the ordinary derived Sullivan unit, hence is rationalization for simply connected spaces of finite rational type \cite{BG76,Sullivan77}. It is therefore an objectwise rational equivalence, and a weak equivalence when $\X$ is already rational. Conversely, let $A$ be fibrant in $\Fun(\C^{\op},\mathcal D^{\op})$, with cohomologically simply connected finite-type values, and choose a projectively cofibrant replacement $\widehat{\mathcal K A}\to\mathcal K A$. The derived counit, written in the algebra direction, is
\[
 A\longrightarrow\A_{\PL}\bigl(\widehat{\mathcal K A}\bigr).
\]
Its evaluations are ordinary derived counits and are quasi-isomorphisms by the same theorem. The ordinary correspondence also shows that the two derived functors preserve the stated simple connectivity and finite-type conditions. The diagram replacements preserve these conditions because their weak equivalences are objectwise.

Thus the derived unit and counit are natural isomorphisms on the indicated localized diagram categories. Restricting first to rational space diagrams proves the equivalence there; functorial objectwise rationalization identifies this category with the rational localization on the left of \eqref{eq:rational-equivalence}. Finally, the algebraic weak equivalences used here and in \cref{prop:model} are the same objectwise quasi-isomorphisms. Localizing the indicated full subcategory of nonnegative diagrams therefore gives the algebraic category in \eqref{eq:rational-equivalence}, independently of the auxiliary model structure used to compute it.
\end{proof}

For a cofibrant model $\M$ of $\A_{\PL}(\X)$, evaluation is cofibrant in ordinary DGCAs, and therefore
\begin{equation}\label{eq:homotopy}
 H^n(\Qind \M)\cong\underline\pi_n(\X)^{\vee}\qquad(n\geq2).
\end{equation}
Duals are taken objectwise. Finite type makes duality exact and involutive on the homotopy representations in this formula.

\subsection{Elementary extensions and Postnikov stages}
Let $\X$ be a simply connected rational diagram, and write $\X_n$ for its $n$th Postnikov section. Put $\underline V=\underline\pi_{n+1}(\X)^\vee$. The next stage is the homotopy fibre of the $k$-invariant:
\begin{equation}\label{eq:postnikov-square}
\begin{tikzcd}[column sep=large,row sep=large]
 \X_{n+1}\arrow[r]\arrow[d]&*\arrow[d]\\
 \X_n\arrow[r,"k^{n+2}"']&K(\underline V^\vee,n+2).
\end{tikzcd}
\end{equation}
This is a homotopy pullback in the rational localization of $\Fun(\C^{\op},\sSet_*)$. The notation $K(\underline W,m)$ means the diagram obtained by applying the Eilenberg--MacLane functor objectwise to a contravariant coefficient system $\underline W$. Rational coefficients are understood in $\underline\pi$.

\Needspace{13\baselineskip}
For one object, a cocycle representative $\widetilde k^{n+2}:\underline V\to Z^{n+2}\M(X_n)$ gives the familiar Hirsch extension. If $\underline V\{r\}$ denotes $\underline V$ placed in degree $r$ with zero differential, it is the pushout
\begin{equation}\label{eq:ordinary-hirsch}
\begin{tikzcd}[column sep=large,row sep=large]
 \Sym(\underline V\{n+2\})\arrow[r,"\widetilde k^{n+2}"]\arrow[d,hook]
 &\M(X_n)\arrow[d,hook]\\
 \Sym(\underline V\{n+1\}\xrightarrow{\id}\underline V\{n+2\})\arrow[r]
 &\M(X_{n+1}).
\end{tikzcd}
\end{equation}
The lower-left algebra is contractible. The new degree-$(n+1)$ generator $v$ satisfies $dv=\widetilde k^{n+2}(v)$. For a general category, $\underline V$ need not be injective, so this extension need not preserve fibrancy. The resolution construction replaces \eqref{eq:ordinary-hirsch} by a natural fibrant version.

For a coefficient system $\underline V$ and an integer $m$, denote by $\Inj^\bullet_m(\underline V)$ the cochain complex obtained from the minimal resolution \eqref{eq:resolution} by putting
\[
 \bigl(\Inj^\bullet_m(\underline V)\bigr)^{m+j}=\Inj^j(\underline V),\qquad d=\partial.
\]
Here $\partial:\Inj^j(\underline V)\to\Inj^{j+1}(\underline V)$ is the resolution differential in \eqref{eq:resolution}, namely the quotient onto the successive cokernel followed by its injective-envelope inclusion. We use the shift convention $K[r]^i=K^{i+r}$ with differential $(-1)^r d_K$. Regard $\underline V$ as a complex in degree zero. The envelope inclusion defines a quasi-isomorphism
\[
 \underline V[-m]\qi\Inj^\bullet_m(\underline V),
\]
since its target has cohomology $\underline V$ in degree $m$ and zero elsewhere. Thus these two complexes are isomorphic objects of $\D(\R)$.

Let $m\geq3$ and let $B$ be a fibrant DGCA diagram satisfying \eqref{eq:reduced}. A morphism
\begin{equation}\label{eq:attachment-class}
 \kappa\in\Hom_{\D(\Vect_k^{\C})}(\underline V[-m],B_+)
 =H^m\RHom_{\Vect_k^{\C}}(\underline V,B_+)
\end{equation}
determines, by the derived free-algebra/augmentation-ideal adjunction, a homotopy class of augmented DGCA maps
\[
 [\Sym(\Inj^\bullet_m(\underline V)),B]_{\Ho(\DGCA_k(\C))}
 \cong\Hom_{\D(\R)}(\underline V[-m],B_+).
\]
The free algebra is cofibrant, and $B_+$ is K-injective. For $k=\Q$ and finite-dimensional $\underline V$, this free algebra models $K(\underline V^\vee,m)$. If $B$ models $\X_n$ and $m=n+2$, the class $\kappa$ corresponding to $k^{n+2}$ in \eqref{eq:postnikov-square} is therefore precisely the algebraic $k$-invariant. Adjoining a null-homotopy of this map will give the homotopy pushout dual to that Postnikov square.

Since $B_+$ is K-injective, $\kappa$ can be represented by a map $\widetilde k^m:\underline V\to Z^mB_+$. We now extend it to the resolution. Injectivity of $B^m$ extends this map across $\underline V\hookrightarrow\Inj(\underline V)$. Its differential vanishes on $\underline V$, so factors through $\im\partial\subset\Inj^1(\underline V)$; injectivity of $B^{m+1}$ extends that factor. Repeating gives
\begin{equation}\label{eq:resolution-ladder}
\begin{tikzcd}[column sep=small,row sep=large]
 \underline V\arrow[r,hook]\arrow[dr,"\widetilde k^m"']
 &\Inj(\underline V)\arrow[r,"\partial"]\arrow[d,dashed,"a^0"]
 &\Inj^1(\underline V)\arrow[r,"\partial"]\arrow[d,dashed,"a^1"]
 &\Inj^2(\underline V)\arrow[r]\arrow[d,dashed,"a^2"]&\cdots\\
 &B^m\arrow[r,"d_B"']&B^{m+1}\arrow[r,"d_B"']&B^{m+2}\arrow[r]&\cdots .
\end{tikzcd}
\end{equation}
At every later step, $d_Ba^j$ vanishes on $\ker(\partial:\Inj^j(\underline V)\to\Inj^{j+1}(\underline V))$, since this kernel is the image of the preceding differential and $d_B^2=0$. Thus the induction is well defined and yields a cochain map
\[
 a:\Inj^\bullet_m(\underline V)\longrightarrow B_+,
 \qquad d_Ba^j=a^{j+1}\partial.
\]
In the Postnikov situation $m=n+2$, the maps $a^j$ are the extended cocycles $\widetilde k^{n+2+j}$ of the elementary-extension construction. Their higher components encode the compatibility across the diagram; they are not additional independent $k$-invariants.

For any complex $K$, define its contractible cone in the present convention by
\[
 \mathrm CK=K\oplus s^{-1}K,\qquad
 d(x)=d_Kx,\qquad d(s^{-1}x)=x-s^{-1}d_Kx.
\]
Here $|s^{-1}x|=|x|-1$. The homotopy $h(x)=s^{-1}x$, $h(s^{-1}x)=0$ satisfies $dh+hd=\id$, and $K\hookrightarrow\mathrm CK$ is a monomorphism of complexes.

\begin{definition}[Elementary extension]\label{def:elementary}
Given $n\geq2$ and $a:\Inj^\bullet_{n+1}(\underline V)\to B_+$, set $K=\Inj^\bullet_{n+1}(\underline V)$. The \emph{elementary extension of $B$ by $\underline V$ in degree $n$} is the pushout
\begin{equation}\label{eq:extension-pushout}
\begin{tikzcd}[column sep=large,row sep=large]
 \Sym K\arrow[r,"\Sym(a)"]\arrow[d,hook]
 &B\arrow[d,hook]\\
 \Sym(\mathrm CK)\arrow[r]
 &B\langle \underline V,n;a\rangle .
\end{tikzcd}
\end{equation}
The top map is the algebra map determined by $a$. Explicitly,
\begin{equation}\label{eq:extension}
 \begin{split}
 B\langle \underline V,n;a\rangle&=B\otimes\Sym(s^{-1}\Inj^\bullet_{n+1}(\underline V)),\\
 D|_B&=d_B,\qquad D(s^{-1}x)=a(x)-s^{-1}\partial x.
 \end{split}
\end{equation}
\end{definition}

Thus $\Inj^j(\underline V)$ supplies generators in degree $n+j$. On those generators the differential has the two parts prescribed by the Postnikov construction: the extended cocycle, with value in $B^{n+1+j}$, and minus the resolution differential, with value in the next group of generators. In particular,
\[
 D^2(s^{-1}x)=d_Ba(x)-a(\partial x)+s^{-1}\partial^2x=0.
\]
For \eqref{eq:postnikov-square}, substitute $n+1$ for the attachment degree $n$: the resolution begins in degree $n+2$, and the new algebra generators begin in degree $n+1$.

\begin{lemma}\label{lem:extension}
Assume the $\otimes$-condition, $n\geq2$, and that $B$ is fibrant and satisfies \eqref{eq:reduced}.
\begin{enumerate}
\item The inclusion $B\hookrightarrow B\langle \underline V,n;a\rangle$ is a cofibration in $\DGCA_k(\C)_{\mathrm{trinj}}$, and its target is fibrant.
\item The quotient of underlying complexes $B\langle \underline V,n;a\rangle/B$ has zero cohomology below $n$, has $H^n=\underline V$, and has $H^{n+1}=0$.
\item Cochain-homotopic maps $a$ give isomorphic extensions over $B$.
\end{enumerate}
\end{lemma}
\begin{proof}
The left arrow of \eqref{eq:extension-pushout} is the image under the left Quillen functor $\Sym$ of a monomorphism of complexes, hence is a cofibration. The same is true of its pushout. The new generators are degreewise injective, so the augmentation ideal of the extension is degreewise injective by the $\otimes$-condition and \cref{lem:tensor}. This proves (i).

Modulo $B$, in degrees $n$ and $n+1$ only linear new generators occur: a product involving a new generator and a positive-degree element has degree at least $n+2$. The differential on these generators, including their differential out of degree $n+1$, is exactly $-\partial$. It has no component in the extra product terms in degree $n+2$. Hence the cohomology in question is that of $s^{-1}\Inj^\bullet_{n+1}(\underline V)$, proving (ii).

If $a-a'=d_Bh+h\partial$, the map from the extension for $a$ to that for $a'$ which fixes $B$ and sends $s^{-1}x$ to $s^{-1}x+h(x)$ commutes with the differentials. Its inverse subtracts $h(x)$. This proves (iii).
\end{proof}

The pushout \eqref{eq:extension-pushout} is a homotopy pushout. To compute it, factor its top arrow as
\[
 \Sym K\lhook\joinrel\longrightarrow\widetilde B
 \xrightarrow{\sim}B,
\]
where the first map is a cofibration and the second an acyclic fibration. This is a cofibrant replacement of $B$ in the category of algebras equipped with a map from $\Sym K$. Since both maps out of the cofibrant algebra $\Sym K$ are now cofibrations, the pushout with $\widetilde B$ computes the homotopy pushout. Its map to \eqref{eq:extension-pushout} is a quasi-isomorphism: filter by the number of new generators; on associated graded objects it is $\widetilde B\to B$ tensored with the symmetric powers of $s^{-1}K$. Induction gives a quasi-isomorphism on every finite filtration stage, and exactness of filtered colimits gives the assertion on the union.

The algebra $\Sym(\mathrm CK)$ is quasi-isomorphic to $\kk$. For $k=\Q$ and finite-dimensional $\underline V$, realization consequently turns \eqref{eq:extension-pushout} into the homotopy-fibre square of the map to $K(\underline V^\vee,n+1)$ represented by \eqref{eq:attachment-class}. This is the diagrammatic generalisation of the Hirsch construction \cite[Section~11]{Scull02}.

We will also use the explicit extension criterion for a model map. Given $\rho:B\to\A$, a map from the new algebra to $\A$ extending $\rho$ is determined by a degree-$(-1)$ map $b:K\to\A_+$ satisfying
\begin{equation}\label{eq:extension-to-target}
 d_{\A}b+b\partial=\rho a.
\end{equation}
It sends $s^{-1}x$ to $b(x)$. Equivalently, $\rho a$ has been supplied with a specified null-homotopy. The universal property is displayed by
\begin{equation}\label{eq:extension-target-diagram}
\begin{tikzcd}[column sep=large,row sep=large]
 \Sym K\arrow[r,"a"]\arrow[d,hook]
 &B\arrow[d,hook]\arrow[ddr,bend left=18,"\rho"]&\\
 \Sym(\mathrm CK)\arrow[r]\arrow[drr,bend right=18,"{(\rho a,b)}"']
 &B\langle \underline V,n;a\rangle\arrow[dr,dashed,"\rho'"]&\\
 &&\A .
\end{tikzcd}
\end{equation}
Both the differential and the model map are therefore controlled by a cochain map and its null-homotopy. This is the input to the inductive construction of geometrically minimal models in the next section, specifically \cref{thm:geometric-existence}.

\Needspace{14\baselineskip}
\section{Geometrically minimal models}\label{sec:geometric}
In this section the good category $\C$ is assumed to satisfy the $\otimes$-condition.
\begin{definition}\label{def:geometric}
A \emph{geometrically minimal diagram} is an increasing union
\[
 \M=\bigcup_{n\geq1}\M(n),\qquad \M(1)=\kk,\qquad
 \M(n)=\M(n-1)\langle \underline V_n,n;a_n\rangle\quad(n\geq2),
\]
where each attachment uses a minimal injective resolution of a finite-dimensional representation $\underline V_n$. The filtration is called the Postnikov filtration. We denote the full category of these diagrams by $\MinModels_k(\C)$. A geometrically minimal model of $A$ is a quasi-isomorphism $\M\to A$ with $A$ fibrant, or the corresponding zigzag if $A$ has not been replaced by a fibrant diagram.
\end{definition}

As a graded algebra, such a diagram is
\begin{equation}\label{eq:geometric-generators}
 \M\cong\Sym\left(\bigoplus_{n\geq2}\bigoplus_{j=0}^{\cdim_k(\C)}\Inj^j(\underline V_n)[-n-j]\right).
\end{equation}
The notation in this formula places a representation with zero differential in the indicated degree. The differential includes both resolution differentials and terms in earlier stages. All sums in a fixed cohomological degree are finite.

\begin{theorem}[Existence]\label{thm:geometric-existence}
Let $\C$ be good and satisfy the $\otimes$-condition, and let $A\in\DGCA_k^{\geq0}(\C)$ be cohomologically simply connected and of finite type. If $A$ is fibrant, there is a direct quasi-isomorphism $\rho:\M\to A$ with $\M$ geometrically minimal. In general there is a zigzag $\M\qi A^{\mathrm f}\xleftarrow{\sim}A$ with $A^{\mathrm f}$ fibrant, giving a model in the sense of \cref{def:geometric}. Geometrically minimal diagrams are cofibrant and fibrant, and each homogeneous component is finite-dimensional.
\end{theorem}
\begin{proof}
First assume $A$ fibrant. Suppose that $\rho:B=\M(n-1)\to A$ induces an isomorphism below $n$ and a monomorphism in degree $n$. Its reduced homotopy-fibre complex is
\[
 T^j=B_+^j\oplus A_+^{j-1},\qquad
 d_T(b,u)=(d_Bb,\rho(b)-d_Au).
\]
It is a bounded-below complex of injectives, with $H^jT=0$ for $j\leq n$. Set $\underline V_n=H^{n+1}T$. The canonical truncation gives the diagram in the derived category
\[
\begin{tikzcd}[column sep=large]
 \underline V_n[-n-1]&\tau_{\leq n+1}T\arrow[l,"\simeq"']\arrow[r]&T.
\end{tikzcd}
\]
It defines a morphism $\underline V_n[-n-1]\to T$ inducing the identity on $H^{n+1}$. Since $T$ is K-injective, this morphism has a cochain representative. Extending it through the minimal resolution by the procedure of \eqref{eq:resolution-ladder} gives a cochain map $\Inj^\bullet_{n+1}(\underline V_n)\to T$, written $(a,b)$. Its equations are
\[
 d_Ba=a\partial,\qquad d_Ab+b\partial=\rho a.
\]
Equation \eqref{eq:extension-to-target} now applies. Attach \eqref{eq:extension} using $a$ and extend $\rho$ by $s^{-1}x\mapsto b(x)$, obtaining $\rho_n:\M(n)\to A$.

This improves the cohomology comparison by one degree. To verify this without choosing separate kernel and cokernel splittings, compare the fibre complexes before and after attachment. Writing $T'$ for the new fibre complex, there is a short exact sequence
\[
 0\longrightarrow T\longrightarrow T'\longrightarrow\M(n)/B\longrightarrow0.
\]
For a new generator represented by $x$, its differential in $T'$ is represented by $(a(x),b(x))$ after the resolution term is removed. Thus the connecting homomorphism is exactly the one prescribed by the map to $T$. By \cref{lem:extension}, its first cohomology is $\underline V_n$ in degree $n$, and the connecting map to $H^{n+1}T$ is precisely the chosen identity. The long exact sequence gives zero fibre cohomology through degree $n+1$. Hence $\M(n)\to A$ is an isomorphism through degree $n$ and a monomorphism in degree $n+1$.

Start with $\M(1)=\kk$ and $n=2$. Filtered colimits are exact, so $\bigcup \M(n)\to A$ is the required direct quasi-isomorphism. For arbitrary $A$, apply this construction to the nonnegative fibrant replacement $A\qi A^{\mathrm f}$ supplied by \cref{lem:connective-replacement}. Cofibrancy follows from the elementary-extension construction. In each degree the construction stabilizes; fibrancy follows from \cref{lem:tensor}. The representations $\underline V_n$ are finite-dimensional by the long exact sequence for $T$ and induction, so \eqref{eq:geometric-generators} proves degreewise finite-dimensionality.
\end{proof}

The fibre construction includes the missing cohomology and the unwanted relations in a single representation:
\[
 0\longrightarrow\coker H^n(\rho)\longrightarrow \underline V_n
 \longrightarrow\ker H^{n+1}(\rho)\longrightarrow0.
\]
The left term consists of cohomology classes not yet represented by the model; the right term consists of classes of the model that must become boundaries in $A$. At each object these are the familiar two steps of the Sullivan construction. The single system $\underline V_n$ keeps their compatibility over $\C$: this extension need not split as a coefficient system.

At the initial stage, $\underline V_2=H^2(A)$ and $a=0$, so
\[
 \M(2)=\Sym(s^{-1}\Inj^\bullet_3(H^2(A))).
\]
The map to $A$ is obtained by extending representatives of $H^2(A)$ along the resolution. More generally, each step fits into the commuting diagram
\[
\begin{tikzcd}[column sep=large]
 \M(n-1)\arrow[r,hook]\arrow[dr,"\rho_{n-1}"']
 &\M(n)\arrow[d,"\rho_n"]\\
 &A.
\end{tikzcd}
\]
The two cohomology conditions in the proof state exactly how much of $A$ has been modelled after that step.

\begin{lemma}\label{lem:postnikov}
For a geometrically minimal diagram $\M$, the complex $\Qind \M(n)$ has cohomology $\underline V_i$ in degrees $2\leq i\leq n$ and zero elsewhere. The inclusion $\Qind \M(n)\to\Qind \M$ identifies it, in the derived category, with the truncation in degrees at most $n$. Consequently, if $\M$ models $\X$, then $\M(n)$ models $\X_n$ and $\underline V_n\cong\underline\pi_n(\X)^{\vee}$.
\end{lemma}
\begin{proof}
There is a short exact sequence
\[
 0\longrightarrow\Qind \M(n-1)\longrightarrow\Qind \M(n)
 \longrightarrow \Inj^\bullet_{n}(\underline V_n)\longrightarrow0,
\]
up to the harmless sign change in the resolution differential. Its connecting map out of $\underline V_n$ lands in $H^{n+1}(\Qind \M(n-1))=0$. Induction proves the cohomology assertion and the truncation statement. Formula \eqref{eq:homotopy} identifies the realized stages with the objectwise Postnikov truncations, naturally in $\C$.
\end{proof}

\begin{lemma}[Straightening]\label{lem:straighten}
Every map $f:\M\to \M'$ of geometrically minimal diagrams is homotopic to a map preserving their Postnikov filtrations.
\end{lemma}
\begin{proof}
This is the analogue of Scull's cocellular approximation \cite[Lemma~13.57]{Scull02}. We construct maps $g_n:\M(n)\to\M'(n)$ and polynomial homotopies
\[
 H_n:\M(n)\longrightarrow P\M',\qquad
 \operatorname{ev}_0H_n=f|_{\M(n)},\qquad
 \operatorname{ev}_1H_n=g_n,
\]
where the last map is followed by $\M'(n)\hookrightarrow\M'$, and require
$H_n|_{\M(n-1)}=H_{n-1}$. Start with the constant homotopy on $\M(1)=\kk$.

Suppose $H_{n-1}$ has been constructed. The evaluation
$\operatorname{ev}_0:P\M'\to\M'$ is a trivial fibration: its kernel is bounded below and degreewise injective, and the constant-path inclusion is a quasi-isomorphism. Since $\M(n-1)\hookrightarrow\M(n)$ is a cofibration, the square
\[
\begin{tikzcd}[column sep=large,row sep=large]
 \M(n-1)\arrow[r,"H_{n-1}"]\arrow[d,hook]
 &P\M'\arrow[d,"\operatorname{ev}_0"]\\
 \M(n)\arrow[r,"f|_{\M(n)}"']\arrow[ur,dashed,"F"]
 &\M'
\end{tikzcd}
\]
has a lift $F$. Put $\psi=\operatorname{ev}_1F$; its restriction to $\M(n-1)$ is $g_{n-1}$.

Write $L=s^{-1}\Inj^\bullet_{n+1}(\underline V_n)$ for the new generator complex, with differential $\ell=-s^{-1}\partial$, so that the algebra differential on $L$ is $\ell+\tau$, where $\tau(L)\subset\M(n-1)$. Put $R_n=\M'/\M'(n)$ and let $q_n:\M'\to R_n$ be the quotient of complexes. This quotient is degreewise injective: in positive degrees $\M'(n)\subset\M'$ splits as a system, since $\M'(n)$ is degreewise injective, and the degree-zero quotient is zero. Moreover, $R_n$ is zero in degrees at most $n$, because all remaining generators have degree at least $n+1$. Thus $R_n$ is K-injective by \cref{lem:kinjective}.

The map $q_n\psi|_L$ is a cochain map, since $\psi\tau$ lands in $\M'(n-1)$. Its homotopy class is zero:
\[
 H^0\Hom_{\R}^{\bullet}(L,R_n)
 \cong\Hom_{\D(\R)}(\underline V_n[-n],R_n)=0.
\]
The last equality follows from the standard truncations. Choose a natural degree-$(-1)$ map $\bar h:L\to R_n$ with
$q_n\psi|_L=d_{R_n}\bar h+\bar h\ell$. A graded splitting of $q_n$ lifts it to $h:L\to\M'_+$, so that
\[
 \psi|_L-(d_{\M'}h+h\ell)
\]
takes values in $\M'(n)$.

We now make this correction while retaining the entire previously chosen homotopy. Regard $h_t(x)=t\,h(x)$ as a degree-$(-1)$ map $L\to P\M'$, and define $H_n$ to equal $H_{n-1}$ on $\M(n-1)$ and, on the new generators, to be
\[
 H_n(x)=F(x)-d_{P\M'}h_t(x)-h_t(\ell x).
\]
The subtracted term is the coboundary of $h_t$ in the Hom complex, hence is a cochain map from $(L,\ell)$ to $P\M'$. It follows that
\[
 d_{P\M'}H_n(x)=H_n(\ell x)+H_{n-1}(\tau x),
\]
so the formula extends multiplicatively to a DGCA map. At $t=0$ the correction vanishes; at $t=1$ it is $d_{\M'}h+h\ell$. Consequently $\operatorname{ev}_0H_n=f|_{\M(n)}$ and $g_n=\operatorname{ev}_1H_n$ takes values in $\M'(n)$. The restriction of $H_n$ to $\M(n-1)$ is exactly $H_{n-1}$.

The compatible maps $H_n$ therefore define a single polynomial homotopy $H:\M\to P\M'$ on the union. Its endpoint $g$ satisfies $g(\M(n))\subset\M'(n)$ for every $n$, proving the lemma.
\end{proof}

\begin{theorem}[Geometric rigidity]\label{thm:geometric-rigidity}
A quasi-isomorphism between geometrically minimal diagrams is homotopic to an isomorphism. Two geometrically minimal models of the same diagram are isomorphic, with their model maps commuting up to homotopy.
\end{theorem}
\begin{proof}
For orbit categories this is \cite[Theorem~3.8 and Corollary~3.9]{Scull02}; the same argument applies here. Straighten the quasi-isomorphism by \cref{lem:straighten}. Applying $\Qind$ gives a quasi-isomorphism, since the source and target are cofibrant. By \cref{lem:postnikov}, its restrictions to the stages and its maps on the successive quotient complexes are quasi-isomorphisms. Those quotients are minimal injective resolutions, so their maps are isomorphisms by \cref{lem:mincomplex}. The filtered map of indecomposables is therefore an isomorphism; the filtration is finite in each degree.

A map between free graded-commutative algebras on generators of degree at least two that is an isomorphism on indecomposables is an isomorphism: invert its linear part and then remove decomposable terms successively by degree. Thus the straightened map is an isomorphism. Finally, model-category lifting provides a quasi-isomorphism between any two cofibrant-fibrant models of the same object, compatible with their maps up to homotopy; apply the first assertion.
\end{proof}

\section{Algebraically minimal models}\label{sec:algebraic}
In classical Sullivan theory, minimal algebras can be described either by successive elementary extensions in increasing generator degrees or by a Sullivan filtration and a decomposable differential \cite[Chapter~12]{FHT01}. The first description provides a construction; the second makes minimality easier to recognize. Their equivalence allows one to construct a model geometrically and verify it algebraically.

For diagrams these approaches lead to \cref{def:geometric,def:algebraic}, and their equivalence can fail. We first prove existence and rigidity for the algebraic notion, then compare the two notions by \cref{thm:spectral-criterion,thm:equivalence}. We continue to assume the $\otimes$-condition. The cancellation lemma itself is valid without any injectivity or tensor hypothesis.
\subsection{Intrinsic minimality}
Suppose $\M=\Sym \underline V$ as a diagram of graded algebras. Its differential has a linear part $d_0:\underline V\to \underline V$, which is the differential on $\Qind \M$, and a decomposable part. The equation $d^2=0$ implies $d_0^2=0$.

\begin{definition}\label{def:algebraic}
An \emph{algebraically minimal diagram} is a cofibrant augmented DGCA diagram $\M$ such that
\begin{enumerate}
\item $\M\cong\Sym \underline V$ as graded algebras, with $\underline V=\bigoplus_{n\geq2}\underline V^n$ degreewise injective and finite-dimensional;
\item the complex $\Qind \M$ is minimal, that is,
\begin{equation}\label{eq:algminimal}
 d_0(\Soc\Qind \M)=0.
\end{equation}
\end{enumerate}
The full category of these diagrams is denoted by $\AlgModels_k(\C)$. Cofibrancy is with respect to \cref{prop:model}; it is an additional requirement on the DGCA diagram, not a consequence of freeness of its underlying graded algebra.
\end{definition}

For an explicit sufficient condition, take an exhaustive filtration starting at $\kk$ whose successive extensions have the form $B\otimes\Sym K$, with
\[
 d(x)=d_Kx+\tau(x),\qquad \tau(K)\subset B_+,
\]
where $K$ is a positive complex and $d_B\tau+\tau d_K=0$. Each inclusion is a pushout of a free cone inclusion, as in \cref{lem:extension}, so the union is cofibrant. Retracts of these diagrams are cofibrant as well. This criterion applies to the geometric construction in \cref{thm:geometric-existence} and to its algebraic retract in \cref{thm:algebraic-existence}.

Condition \eqref{eq:algminimal} is independent of a choice of generators. In a chosen presentation it is equivalent to
\begin{equation}\label{eq:decomposable}
 d\left(\bigcap_{f:c\to d\text{ noninvertible}}\ker \M(f)\right)
 \subset (\M_+(c))^2\qquad(c\in\C).
\end{equation}
Here $(\underline V^n)_c$ denotes the kernel system component from \eqref{eq:socle}, degree by degree. Indeed, the linear part of an element in the intersection lies in $\underline V_c$, and the differential preserves the decomposable ideal. Conversely, the natural inclusion $\underline V\to\Sym \underline V$ lifts every element of $\underline V_c$ to that intersection. At an object with no outgoing nonisomorphisms, \eqref{eq:decomposable} says $d\M(c)\subset\M_+(c)^2$, exactly the ordinary Sullivan minimality condition recalled in \cref{sec:intro}; see \cite[Chapter~12]{FHT01}.

For orbit categories, the kernel condition \eqref{eq:decomposable} occurs in Triantafillou's definition \cite[Definition~5.1]{Triantafillou82}, whose freeness requirement is objectwise. Our definition also requires natural free generators and cofibrancy. Scull's example \cite[Section~21]{Scull02}, reproduced over a square in \cref{subsec:square}, shows that a model satisfying the earlier algebraic condition need not be geometrically minimal. Consequently the elementary-extension argument does not establish algebraic existence as asserted in \cite[Theorem~5.3]{Triantafillou82}: geometric construction alone cannot ensure \eqref{eq:decomposable}. The cancellation argument below establishes existence with our stronger hypotheses. The discrepancy concerns the two constructions; it does not contradict the existence of algebraically minimal models.

\Needspace{18\baselineskip}
\subsection{Cancellation of contractible generators}
For ordinary Sullivan algebras, decomposition into a minimal factor and a contractible factor is \cite[Theorem~14.9]{FHT01}. The following diagram version cancels any naturally contractible summand of the linear complex; its proof makes the necessary nonlinear change of generators explicit.

\begin{lemma}\label{lem:cancellation}
Let $\M=(\Sym \underline V,d)$ be a diagram of augmented DGCAs with generators in degrees at least two. Suppose that its linear differential admits a decomposition of complexes
\[
 (\underline V,d_0)=(\underline W,d_{\underline W})\oplus(C,d_C),
\]
and that $C$ is contractible by a natural contracting homotopy. Then there is a degree-one derivation $D_{\underline W}$ on $\Sym \underline W$, with linear part $d_{\underline W}$, and an isomorphism of DGCA diagrams
\begin{equation}\label{eq:cancellation}
 (\Sym \underline V,d)\cong(\Sym \underline W,D_{\underline W})\otimes(\Sym C,d_C).
\end{equation}
Under the identification $\underline V=\underline W\oplus C$, the coordinate change $\Phi$ inducing this isomorphism satisfies
\[
 \Phi(v)-v\in\Sym^{\geq2}\underline V\quad(v\in\underline V),
 \qquad \Qind(\Phi)=\id_{\underline V}.
\]
\end{lemma}
\begin{proof}
For homogeneous operators $u,v$, use the graded commutator
\[
 [u,v]=uv-(-1)^{|u||v|}vu.
\]
A derivation has \emph{arity $r$} if it takes generators to words of length $r$. Such a derivation is uniquely determined by its restriction to $\underline V$. The arity-$r$ derivation complex is therefore
\[
 \Der_r^a=\Hom_{\R}^a(\underline V,\Sym^r\underline V),
 \qquad \partial_{\Der}(f)=[d_0,f]
 =d_0f-(-1)^a f d_0.
\]
Here $\Hom_{\R}^a$ consists of natural graded maps raising cohomological degree by $a$, and $d_0$ on the target is its derivation extension. Filter derivations by arity and write
\[
 d=d_0+\delta_2+\delta_3+\cdots,
 \qquad \delta_r\in\Hom_{\R}(\underline V,\Sym^r\underline V)^1.
\]
Let $h_C$ be a natural contraction of $C$, and extend it to a degree-$(-1)$ derivation $h$ of $\Sym\underline V$ by setting $h|_{\underline W}=0$. The commutator $[d_0,h]=d_0h+hd_0$ is a derivation, zero on $\underline W$ and the identity on $C$. It consequently acts as multiplication by $j$ on
\[
 \Sym^{r-j}\underline W\otimes\Sym^j C.
\]
The kernel $J_r=\ker(\Sym^r\underline V\to\Sym^r\underline W)$ is the sum of these subcomplexes for $j>0$. On its $j$th summand, $h/j$ gives a contraction $h_{J_r}$. Characteristic zero is crucial here.

With differential $[d_0,-]$, the arity-$r$ derivation complex splits as
\begin{equation}\label{eq:derivation-splitting}
 \Hom(\underline W,\Sym^r\underline W)\ \oplus\ \Hom(\underline W,J_r)\ \oplus\ \Hom(C,\Sym^r\underline V).
\end{equation}
All Hom spaces in this formula are complexes of natural graded maps. On the last two summands the degree-$(-1)$ contractions are, respectively,
\[
 f\longmapsto h_{J_r}f,\qquad
 f\longmapsto(-1)^{|f|}f h_C.
\]
Substituting into $\partial_{\Der}H+H\partial_{\Der}$ gives the identity. Thus the complement of the derivations involving only $\underline W$ is contractible, without any projectivity requirement on $\underline W$.

Assume inductively that the terms of arity less than $r$ involve only $\underline W$. In the arity-$r$ component of $d^2=0$, all expressions involving those lower terms also involve only $\underline W$. Hence the complementary component of $\delta_r$ is a $[d_0,-]$-cocycle. By \eqref{eq:derivation-splitting}, it equals $[d_0,\xi_r]$ for a degree-zero derivation $\xi_r$ of arity $r$. Replacing $d$ by $\exp(\xi_r)d\exp(-\xi_r)$ removes that component and changes no lower-arity term. Continuing gives a differential $D_{\underline W}+d_C$.

The argument takes place in ordinary graded algebras. A monomial of cohomological degree $m$ has at most $\lfloor m/2\rfloor$ factors. Every degree-zero derivation used above raises word length, so its exponential and the infinite sequence of coordinate changes are finite on each homogeneous component. The resulting conjugation is a well-defined natural algebra automorphism, with inverse obtained in the same way.
\end{proof}

\begin{theorem}[Existence and tensor decomposition]\label{thm:algebraic-existence}
Let $\M$ be a cofibrant diagram whose underlying graded algebra is $\Sym \underline V$, with $\underline V$ degreewise finite-dimensional, injective, and concentrated in degrees at least two. There are an algebraically minimal diagram $\M_{\mathrm{alg}}$ and a contractible complex $C$ of injectives such that
\begin{equation}\label{eq:main-decomposition}
 \M\cong \M_{\mathrm{alg}}\otimes\Sym C.
\end{equation}
Both the inclusion and projection associated with this tensor decomposition are quasi-isomorphisms. In particular, every diagram in \cref{thm:geometric-existence} has an algebraically minimal model.
\end{theorem}
\begin{proof}
Apply \cref{lem:mincomplex} to $\Qind \M$ to split it as a minimal complex $\underline W$ and a sum $C$ of injective disks. Then apply \cref{lem:cancellation}. The resulting factor $(\Sym \underline W,D_{\underline W})$ has minimal indecomposable complex by construction. It is a retract of $\M$ as a DGCA diagram, so it is cofibrant. Its generators are injective summands of the generators of $\M$, and it is fibrant by \cref{lem:tensor}. The augmentation of $\Sym C$ is a quasi-isomorphism: the same counting contraction used in the cancellation lemma contracts its augmentation ideal. This proves the assertions. For a general $A$, first apply \cref{thm:geometric-existence}.
\end{proof}

The summands in \eqref{eq:main-decomposition} are not obtained merely by projecting the original differential onto a chosen complement of $C$. The coordinate change in \cref{lem:cancellation} is what makes that complement an actual differential subalgebra.

\begin{theorem}[Algebraic rigidity]\label{thm:algebraic-rigidity}
Every quasi-isomorphism between algebraically minimal diagrams is an isomorphism. Algebraically minimal models are unique up to isomorphism, compatibly with their model maps up to homotopy.
\end{theorem}
\begin{proof}
For a quasi-isomorphism $f:\M\to N$, the map $\Qind f$ is a quasi-isomorphism between bounded-below minimal complexes of injectives. It is an isomorphism by \cref{lem:mincomplex}. Hence $f$ is an isomorphism of free graded algebras, by induction on degree as in \cref{thm:geometric-rigidity}, and consequently of DGCAs. Existence of a comparison quasi-isomorphism between two models follows from cofibrancy and fibrancy.
\end{proof}

Algebraic rigidity is stronger than \cref{thm:geometric-rigidity}: the given quasi-isomorphism is already an isomorphism, whereas a geometric quasi-isomorphism may first require a homotopy.

\begin{lemma}[Algebraic models as deformation retracts]\label{lem:algebraic-retract}
Let $A$ be a fibrant, cohomologically simply connected DGCA diagram of finite type. There are a geometric model $\rho:\M_{\mathrm{geom}}\qi A$ and an algebraically minimal subdiagram $\M_{\mathrm{alg}}\subset\M_{\mathrm{geom}}$ with a retraction $p$ such that
\[
 p i=\id_{\M_{\mathrm{alg}}},\qquad
 i p\simeq\id_{\M_{\mathrm{geom}}}\ \text{relative to }\M_{\mathrm{alg}}.
\]
In particular, $\rho i:\M_{\mathrm{alg}}\qi A$ is a direct algebraic model map. For arbitrary $A$ the same assertion holds after a fibrant replacement.
\end{lemma}
\begin{proof}
Apply \cref{thm:geometric-existence,thm:algebraic-existence} to identify
\[
 \M_{\mathrm{geom}}\cong\M_{\mathrm{alg}}\otimes\Sym C,
\]
where $C$ is a sum of injective disks. Transport the inclusion and augmentation projection through this isomorphism. On each disk write its generators as $u,v$ with $du=v$ and $dv=0$. A natural polynomial homotopy from $ip$ to the identity fixes $\M_{\mathrm{alg}}$ and is determined by
\[
 H(u)=tu,\qquad H(v)=d(tu)=dt\,u+tv.
\]
These formulas respect the differential and give $H|_{t=0}=ip$, $H|_{t=1}=\id$; the retraction remains constant throughout. They apply to the whole generator systems, so the homotopy is natural in $\C$. The inclusion $i$ is a quasi-isomorphism by \cref{thm:algebraic-existence}, proving the assertion about $\rho i$.
\end{proof}

\Needspace{10\baselineskip}
\begin{corollary}[Classification of rational homotopy types]\label{cor:classification}
Let $\MinModels_{\Q}(\C)$ and $\AlgModels_{\Q}(\C)$ denote the full categories of geometrically and algebraically minimal diagrams over $\Q$. There are equivalences
\begin{align*}
 (\MinModels_{\Q}(\C)/\simeq_{\mathrm h})^{\op}
 &\simeq\Ho(\Spaces),\\
 (\AlgModels_{\Q}(\C)/\simeq_{\mathrm h})^{\op}
 &\simeq\Ho(\Spaces).
\end{align*}
Here $\simeq_{\mathrm h}$ identifies homotopic maps. In each class, isomorphism classes of objects correspond bijectively to rational diagram homotopy types.
\end{corollary}
\begin{proof}
By \cref{thm:geometric-existence,lem:algebraic-retract}, every indicated algebraic homotopy type is represented by an object of each class. All these objects are cofibrant and fibrant, so maps in the localized category are precisely their homotopy classes of maps. The respective rigidity theorems identify weak-equivalence classes with isomorphism classes. Apply \eqref{eq:rational-equivalence}.
\end{proof}

\begin{lemma}[Homotopy automorphisms]\label{lem:homotopy-automorphisms}
Let $\C$ be good and satisfy the $\otimes$-condition, and let $\X$ be a pointed diagram of simply connected spaces of finite rational type. Let $\M$ be either a geometrically or an algebraically minimal model of $\X$. Write $\mathcal E_{\Q}(\X)$ for its automorphism group in the rational diagram homotopy category, and $\Aut(\M)$ for the group of natural augmented DGCA automorphisms. Then
\[
 \Aut_0(\M)=\{\phi\in\Aut(\M)\mid\phi\simeq\id_{\M}\}
\]
is a normal subgroup, where the homotopies are natural augmented DGCA homotopies, and
\begin{equation}\label{eq:homotopy-automorphisms}
 \mathcal E_{\Q}(\X)^{\op}\cong\Aut(\M)/\Aut_0(\M).
\end{equation}
The identification is relative to the chosen model equivalence. In particular, every rational homotopy automorphism is represented contravariantly by a strict automorphism of $\M$.
\end{lemma}
\begin{proof}
By \cref{cor:classification}, cofibrancy and fibrancy identify $\mathcal E_{\Q}(\X)^{\op}$ with the invertible homotopy classes of endomorphisms of $\M$. Any representative of such a class is a quasi-isomorphism. By \cref{thm:geometric-rigidity,thm:algebraic-rigidity}, it is homotopic to an isomorphism in the geometric case and is already an isomorphism in the algebraic case. Thus the homomorphism from $\Aut(\M)$ onto this group is surjective, with kernel $\Aut_0(\M)$.
\end{proof}

The quotient in \eqref{eq:homotopy-automorphisms} uses homotopies of the whole diagram; separate homotopies at its objects need not assemble into one. \Cref{ex:tubular-automorphisms,ex:projective-automorphisms} compute this quotient and identify which boundary self-equivalences preserve the embedding data.

\section{Comparison of the two minimality conditions}\label{sec:comparison}
For a single space, the geometric and algebraic descriptions of Sullivan minimality coincide. Scull's equivariant example \cite[Section~21]{Scull02} shows that this need not persist for diagrams. Their agreement depends both on the indexing category and on the particular rational homotopy type. The exact obstruction is given by \cref{thm:spectral-criterion}; \cref{cor:gaps} gives sufficient conditions in terms of injective dimensions and homotopy degrees. Finally, \cref{thm:equivalence} characterizes the categories for which every type has both minimality properties.

Throughout this section $\C$ is good and satisfies the $\otimes$-condition. For statements involving spaces take $k=\Q$; the algebraic criteria work over any field of characteristic zero.
\begin{definition}\label{def:algebraic-type}
A rational diagram homotopy type $\X$ is \emph{algebraic} if its geometrically minimal model is algebraically minimal.
\end{definition}
This terminology concerns the compatibility of the two minimality conditions. Every homotopy type has an algebraically minimal model, whether or not it is algebraic in this sense. The definition is independent of the chosen geometric model by \cref{thm:geometric-rigidity}.

\subsection{An intrinsic obstruction}
Put
\[
 L_{\X}=\LQ \A_{\PL}(\X),\qquad \underline V_q=\underline\pi_q(\X)^{\vee}.
\]
For a geometric model $\M$, the complex $\Qind \M$ represents $L_{\X}$ and has cohomology $\underline V_q$. Its graded components are the sums of the $\Inj^p(\underline V_q)$ in total degree $p+q$.

\begin{theorem}[Hyper-Ext criterion]\label{thm:spectral-criterion}
For each simple $S\in\R$, there is a strongly convergent spectral sequence
\begin{equation}\label{eq:hyperext}
 E_2^{p,q}=\Ext^p_{\R}(S,\underline V_q)
 \Longrightarrow H^{p+q}\RHom_{\R}(S,L_{\X}),
 \qquad d_r:E_r^{p,q}\to E_r^{p+r,q-r+1}.
\end{equation}
The type $\X$ is algebraic if and only if \eqref{eq:hyperext} collapses at $E_2$ for every simple $S$.
\end{theorem}
\begin{proof}
Use the Postnikov filtration of $\Qind \M$, which represents its canonical truncation filtration by \cref{lem:postnikov}. Applying $\Hom(S,-)$ gives the usual hyper-Ext spectral sequence. More explicitly, filter by injective-resolution degree $p$. The internal resolution differential raises $p$ by one. A linear attachment from stage $q$ to an earlier stage $q'<q$ raises $p$ by $q+1-q'\geq2$. Thus this is a filtration by subcomplexes, with
\[
 E_2^{p,q}=\Hom_{\R}(S,\Inj^p(\underline V_q))=\Ext^p_{\R}(S,\underline V_q),
\]
since the resolutions are minimal. It is bounded by $0\leq p\leq \cdim_k(\C)$, so convergence is strong. Since $\Qind \M$ is a bounded-below complex of injectives, the abutment has the stated derived interpretation.

If $\M$ is algebraically minimal, the differential on $\Hom(S,\Qind \M)$ is zero, and the spectral sequence collapses. Conversely, collapse implies, in every total degree $m$,
\[
 \dim H^m\Hom(S,\Qind \M)=\dim\Hom(S,(\Qind \M)^m).
\]
Both dimensions are finite. The elementary formula
\[
 \dim H^mK=\dim K^m-\operatorname{rank}d^{m-1}-\operatorname{rank}d^m
\]
therefore forces the differential on $\Hom(S,\Qind \M)$ to vanish. Varying $S$ shows that $d_0$ vanishes on the entire socle of $\Qind \M$, which is algebraic minimality.
\end{proof}

\begin{lemma}[Permanence of algebraicity]\label{lem:algebraicity-permanence}
Algebraicity is invariant under rational equivalences of diagrams. It is preserved by homotopy retracts, finite products, and Postnikov sections. It is also preserved by products $\prod_i\X_i$ for which, for each $d$, only finitely many factors have nonzero rational homotopy in degrees at most $d$.
\end{lemma}
\begin{proof}
The complex $L_{\X}$ and its canonical truncation filtration are rational homotopy invariants, so \cref{thm:spectral-criterion} proves the first assertion. A homotopy retraction of diagrams induces a retraction of these filtered derived objects and hence of their hyper-Ext spectral sequences. Every differential of the retract is a retract of the corresponding differential; collapse therefore passes to homotopy retracts.

A product is modelled by the tensor product of cofibrant models. Indecomposables turn this tensor product into a direct sum, giving
\[
 L_{\prod_i\X_i}\simeq\bigoplus_i L_{\X_i}.
\]
The spectral sequence \eqref{eq:hyperext} is consequently the direct sum of those for the factors. For the stated infinite products, the tensor products of geometric models stabilize degree by degree; only finitely many summands contribute to any total degree, so the same argument applies. Finally, the geometric model of a Postnikov section is a stage $\M(n)$. Its indecomposables form a subcomplex of $\Qind\M$, and the socle of that subcomplex is contained in the full socle. Thus \eqref{eq:algminimal} passes to $\M(n)$.
\end{proof}

There is also a useful sufficient condition at a single stage. Projecting an attachment to indecomposables gives a class
\begin{equation}\label{eq:linear-k}
 \lambda_n\in\Hom_{\D(\R)}(\underline V_n[-n-1],\Qind \M(n-1)).
\end{equation}
If this class vanishes, its linear representative can be removed by a linear change of the new generators into old ones. If this happens at every stage, the indecomposable complex is the direct sum of the minimal resolutions $\Inj^\bullet_{n}(\underline V_n)$, and the type is algebraic. For example, the Eilenberg--MacLane diagram of \cref{ex:arrow-em} has exactly one such summand. Products of Eilenberg--MacLane diagrams have one summand for each factor, with all attachment classes zero.

\Needspace{10\baselineskip}
\begin{corollary}[Dimension and gap criteria]\label{cor:gaps}
Each of the following conditions implies that $\X$ is algebraic:
\begin{enumerate}
\item $\operatorname{id}_{\R}\underline V_n\leq1$ for every $n$;
\item whenever $\underline V_m,\underline V_n\ne0$ and $m<n$, one has $\operatorname{id}_{\R}\underline V_m\leq n-m$.
\end{enumerate}
In particular, every type is algebraic if $\cdim_k(\C)\leq1$. More generally, gaps of at least $\cdim_k(\C)$ between nonzero homotopy degrees suffice.
\end{corollary}
\begin{proof}
Condition (i) implies (ii). Under (ii), all generators of the earlier stage associated with $\underline V_m$ have degrees at most $m+\operatorname{id}\underline V_m\leq n$. The attachment of degree-$n$ generators takes values in degrees at least $n+1$, so it has no linear component in an earlier stage. Therefore $\Qind \M$ is the direct sum of its minimal resolutions.
\end{proof}

Ordinary Sullivan minimal models are not strictly functorial: choosing models and representatives of maps objectwise does not ensure compatibility with every composition and relation in $\C$. The next criterion gives diagrams for which all this compatibility can be achieved by ordinary minimal Sullivan algebras. It is a sufficient condition for such strict representatives; the necessity below concerns the objectwise minimality of the geometric model itself.

\begin{proposition}[Objectwise minimality and rectification]\label{prop:objectwise}
A geometrically minimal model of $\X$ is an ordinary minimal Sullivan algebra at every object if and only if every $\underline V_n=\underline\pi_n(\X)^{\vee}$ is injective. Under this condition the geometric model is a strictly commutative diagram of ordinary minimal Sullivan models.
\end{proposition}
\begin{proof}
If every $\underline V_n$ is injective, no resolution generators occur in positive resolution degree. The differential of a degree-$n$ generator lands in the algebra generated in degrees below $n$ and is decomposable. Conversely, objectwise minimality means that the complex $\Qind \M$ has zero differential. Hence $\underline V_n=H^n(\Qind \M)=(\Qind \M)^n$ is injective. Duality gives the projective formulation. The structure maps of $\M$ and its natural model map already provide the asserted strict representatives.
\end{proof}

For an Eilenberg--MacLane diagram $K(\underline V^\vee,n)$, put $I=\Inj^\bullet_n(\underline V)$, with its resolution differential $d_I$. Both minimal systems have the following value on each arrow $f:c\to a$:
\[
\begin{tikzcd}[column sep=large]
 (\Sym I(c),d_I)\arrow[r,"\Sym I(f)"]
 &(\Sym I(a),d_I).
\end{tikzcd}
\]
Thus every such diagram is algebraic, as are its products by \cref{lem:algebraicity-permanence}.

\subsection{The commutative square}\label{subsec:square}
Let $\C$ be the poset with arrows $4\to2\to1$ and $4\to3\to1$, with the two composites equal. Write $S_c$ for the one-dimensional representation supported at $c$, and $\Ic{c}=\coind{c}{k}$. The minimal injective resolution of $S_1$ is
\begin{equation}\label{eq:square-resolution}
 0\longrightarrow S_1\longrightarrow \Ic{1}
 \xrightarrow{(1,1)}\Ic{2}\oplus \Ic{3}
 \xrightarrow{(1,-1)}\Ic{4}\longrightarrow0.
\end{equation}
At $4$ these are the maps $k\to k^2\to k$ displayed in the formula; at $2,3$ the nonzero first maps are identities. Thus $\cdim_k(\C)=2$, and the $\otimes$-condition holds by \cref{ex:meet}.

For later computations, \cref{prop:envelope} gives the injective envelope of any system $F$ on this square explicitly:
\begin{equation}\label{eq:tubular-envelope}
 \begin{split}
 \Inj(F)&\cong\bigl(\Ic{1}\otimes F(1)\bigr)
 \oplus\bigl(\Ic{2}\otimes\ker F_{21}\bigr)
 \oplus\bigl(\Ic{3}\otimes\ker F_{31}\bigr)\\
 &\qquad\oplus\bigl(\Ic{4}\otimes(\ker F_{42}\cap\ker F_{43})\bigr).
 \end{split}
\end{equation}
Here $F_{ij}:F(i)\to F(j)$ is a structure map. Applying this formula to successive cokernels computes the minimal resolution, whose length is at most two.

Consider the two-term complex $\Ic{1}$ in degree $3$ and $\Ic{2}\oplus \Ic{3}$ in degree $4$, with the first differential in \eqref{eq:square-resolution}. Its free algebra $\M_{\mathrm{alg}}$ is the square
\begin{equation}\label{eq:square-algebra}
\begin{tikzcd}[column sep={3.2cm,between origins},row sep=large]
 &\begin{gathered}4:\ \Sym(x_3,y_4,z_4)\\dx=y+z\end{gathered}
 \arrow[dl,"z\mapsto0"']\arrow[dr,"y\mapsto0"]&\\
 \begin{gathered}2:\ \Sym(x_3,y_4)\\dx=y\end{gathered}
 \arrow[dr,"y\mapsto0"']&&
 \begin{gathered}3:\ \Sym(x_3,z_4)\\dx=z\end{gathered}
 \arrow[dl,"z\mapsto0"]\\
 &\begin{gathered}1:\ \Sym(x_3)\\d=0\end{gathered}&
\end{tikzcd}
\end{equation}
Unspecified differentials vanish; arrows preserve identically named generators and send the others to zero. The algebra is cofibrant because it is free on a complex. Its indecomposable complex is minimal: the socle generators occur at $1,2,3$ and their differentials vanish there. Hence $\M_{\mathrm{alg}}$ is algebraically minimal. Its homotopy duals are
\[
 H^3\Qind \M_{\mathrm{alg}}=S_1,\qquad H^4\Qind \M_{\mathrm{alg}}=S_4.
\]

The geometrically minimal model is the full DGCA diagram
\begin{equation}\label{eq:square-geometric}
\begin{tikzcd}[column sep={3.2cm,between origins},row sep=large]
 &\begin{gathered}4:\ \Sym(x_3,y_4,z_4,v_4,w_5)\\
 dx=y+z,\quad dy=dv=w,\quad dz=-w\end{gathered}
 \arrow[dl,"{z,v,w\mapsto0}"']\arrow[dr,"{y,v,w\mapsto0}"]&\\
 \begin{gathered}2:\ \Sym(x_3,y_4)\\dx=y\end{gathered}
 \arrow[dr,"y\mapsto0"']&&
 \begin{gathered}3:\ \Sym(x_3,z_4)\\dx=z\end{gathered}
 \arrow[dl,"z\mapsto0"]\\
 &\begin{gathered}1:\ \Sym(x_3)\\d=0\end{gathered}&
\end{tikzcd}
\end{equation}
Its first stage uses the complete resolution \eqref{eq:square-resolution} in degrees $3,4,5$; the next adjoins $v_4$ at $4$ with $dv=w$. Since $v$ belongs to the socle and has nonzero linear differential, this model is not algebraically minimal.

The cancellation is completely explicit. Put $Y=y-v$ and $Z=z+v$ at $4$; these substitutions are natural because $v$ is supported at $4$. Then
\[
 dY=dZ=0,\qquad dx=Y+Z,\qquad dv=w,
\]
and
\begin{equation}\label{eq:square-cancellation}
 \M_{\mathrm{geom}}\cong \M_{\mathrm{alg}}\otimes\Sym(v_4,w_5),\qquad dv=w.
\end{equation}
This is a homotopy type that has an algebraically minimal model but is not algebraic in the sense of \cref{def:algebraic-type}. The nonzero class
\[
 \Ext^2_{\R}(S_4,S_1)\cong k
\]
is its linear Postnikov attachment. In \eqref{eq:hyperext} for $S=S_4$, the differential from the degree-four homotopy term to the second injective-resolution term of $S_1$ is an isomorphism. It detects exactly the pair cancelled in \eqref{eq:square-cancellation}.

\section{Splitting H-space diagrams}\label{sec:splitting}
A rational H-space splits into Eilenberg--MacLane spaces, but an equivariant splitting can fail. Triantafillou proved splitting for cyclic groups of prime-power order and constructed nonsplit Hopf $C_p\times C_q$-spaces for distinct primes $p,q$ \cite{Triantafillou83}; see also her account in \cite[Chapter~III, Theorem~3.1]{May96}. We extend this phenomenon to all good indexing categories: universal splitting holds exactly when the category algebra is hereditary (\cref{thm:h-splitting}). Under the $\otimes$-condition, this is also equivalent to every rational diagram homotopy type being algebraic (\cref{thm:equivalence}).

In this section $k=\Q$. The splitting theorem itself does not require the $\otimes$-condition. Put $\B=\Fun(\C^{\op},\Vect_{\Q})$.

\begin{definition}\label{def:hobject}
An \emph{H-diagram} is an H-object in the pointed homotopy category of diagrams: it has a multiplication $\mu:\X\times \X\to \X$ for which the two inclusions of $\X$ are units up to homotopy in that category. Associativity and commutativity are not assumed.
\end{definition}
This is a compatibility condition on the diagram, not merely a choice of unrelated H-space structures at its objects. Its Postnikov truncations inherit multiplications, since truncation preserves products.

\subsection{Coefficients and power maps}
For $\underline W\in\B$, use reduced diagram cohomology
\[
 \widetilde H^m_{\C}(Y;\underline W)
 =[Y,K(\underline W,m)]_*
 =H^m\RHom_{\B}(\widetilde C_*(Y;\Q),\underline W).
\]
In the last expression homological chains are given their usual cohomological Hom grading. One can compute with chains of a projectively cofibrant replacement of $Y$; those chains are projective representations. Resolving $\underline W$ gives the natural universal coefficient spectral sequence
\begin{equation}\label{eq:uct}
 E_2^{p,q}=\Ext^p_{\B}(\widetilde H_q(Y;\Q),\underline W)
 \Longrightarrow\widetilde H^{p+q}_{\C}(Y;\underline W).
\end{equation}
Its filtration is finite because $\cdim_{\Q}(\C)<\infty$. If $\cdim_{\Q}(\C)\leq1$, it yields the natural short exact sequence
\begin{equation}\label{eq:uct-short}
\begin{split}
0\longrightarrow\Ext^1_{\B}(\widetilde H_{m-1}(Y),\underline W)
\longrightarrow\widetilde H^m_{\C}(Y;\underline W)\\
\longrightarrow\Hom_{\B}(\widetilde H_m(Y),\underline W)
\longrightarrow0.
\end{split}
\end{equation}

\Needspace{9\baselineskip}
\begin{lemma}[Power-map weights]\label{lem:weights}
Let $Y$ be a simply connected rational H-space of finite type, and let $T=[2]^*$ on $H^*(Y;\Q)$, where $[2]=\mu\circ\Delta$.
\begin{enumerate}
\item The minimal Sullivan differential of $Y$ is zero. Consequently, $H^*(Y)$ is free graded-commutative, with indecomposables dual to $\pi_*(Y)$.
\item For its augmentation ideal $J$, the map induced by $T$ on $J^r/J^{r+1}$ is multiplication by $2^r$.
\item If $\pi_i(Y)=0$ for $i\geq n$, then $T-2$ is invertible on $H^n(Y)$ and $H^{n+1}(Y)$. These inverses are natural for maps commuting with the power maps.
\end{enumerate}
\end{lemma}
\begin{proof}
Let $(\Sym U,d)$ be the ordinary minimal Sullivan algebra. A representative for multiplication is a map $\mu^*:\Sym U\to\Sym U\otimes\Sym U$. The unit identities imply that its linear part is $u\mapsto u\otimes1+1\otimes u$: homotopic maps of ordinary minimal algebras induce identical maps on indecomposables.

Suppose $n$ is the lowest generator degree in which $d$ is nonzero, and choose $u\in U^n$ with $du\ne0$. Every nonlinear term in $\mu^*(u)$ involves generators of degrees less than $n$, so has zero differential. Hence
\[
 \mu^*(du)=du\otimes1+1\otimes du.
\]
Take the lowest nonzero word-length term of $du$. It has length at least two and is primitive for the coproduct making $U$ primitive, since all nonlinear corrections to $\mu^*$ raise word length. A free graded-commutative algebra in characteristic zero has no such primitive: the component of the reduced coproduct with one generator in the first factor is its graded derivative, and multiplication of that component gives its word length times the element. This contradiction proves (i).

The unit identities imply $T=2$ on $J/J^2$. Since $T$ is an algebra map, it is $2^r$ on the $r$th associated graded piece. If $Y$ has no homotopy in degrees at least $n$, the components of degrees $n,n+1$ are in $J^2$. Their finite word-length filtrations have weights $2^r$ with $r\geq2$, none equal to $2$. This proves invertibility. It can equivalently be expressed as a polynomial in $T$ in each degree, which proves naturality.
\end{proof}

\begin{theorem}[Rational splitting]\label{thm:h-splitting}
Let $\C$ be a good category. The following conditions are equivalent:
\begin{enumerate}
\item $\gldim(\Q\C)\leq1$;
\item every pointed, objectwise simply connected H-diagram of finite rational type admits a rational equivalence
\begin{equation}\label{eq:h-splitting}
 \X_{\Q}\simeq\prod_{n\geq2}K(\underline\pi_n(\X),n);
\end{equation}
\item every simply connected diagram of simplicial rational vector spaces of finite type having exactly two nonzero homotopy representations splits in this way.
\end{enumerate}
The splitting in (ii) is an equivalence of underlying diagrams; it is not asserted to preserve the given multiplication or to be canonical.
\end{theorem}
\begin{proof}
Assume (i) and rationalize $\X$. Write $Y=\X_{n-1}$ and $\underline W=\underline\pi_n(\X)$. The next stage is classified by
\[
 k_n\in\widetilde H^{n+1}_{\C}(Y;\underline W).
\]
The power map of $\X$ induces the power map on $Y$ and multiplication by $2$ on $\underline W$. Naturality of Postnikov invariants gives
\begin{equation}\label{eq:power-k}
 [2]^*k_n=2k_n.
\end{equation}
Apply \eqref{eq:uct-short} with $m=n+1$. Objectwise, $Y$ is a rational H-space with homotopy concentrated below $n$. By \cref{lem:weights}, the operator $[2]_*-2$ is an invertible natural map on $H_n(Y)$ and $H_{n+1}(Y)$; duality is valid by finite type. Therefore $[2]^*-2$ is invertible on both outer terms of \eqref{eq:uct-short}, and hence on its middle term. Equation \eqref{eq:power-k} forces $k_n=0$.

This holds for every $n\geq2$. Choose a splitting of each zero Postnikov extension over the preceding stage. These choices give compatible equivalences of towers. Passing to the homotopy inverse limit gives \eqref{eq:h-splitting}, since a simply connected space is the homotopy inverse limit of its Postnikov tower and the product has only finitely many factors contributing to each homotopy degree. This proves (ii). The implication (ii)$\Rightarrow$(iii) is immediate.

For the converse, suppose $\gldim(\Q\C)>1$. There are finite-dimensional $\underline V,\underline W\in\B$ and a nonzero
\[
 e\in\Ext^2_{\B}(\underline V,\underline W).
\]
Indeed, vanishing of all such groups is equivalent to every first syzygy being projective, hence to heredity. Represent $e$ by
\begin{equation}\label{eq:twoextension}
 0\longrightarrow \underline W\longrightarrow E_1\xrightarrow{d}E_0\longrightarrow \underline V\longrightarrow0.
\end{equation}
For $n\geq3$, place $E_1$ in homological degree $n+1$ and $E_0$ in degree $n$, and apply Dold--Kan objectwise. This gives a diagram $Z$ of simplicial rational vector spaces with $\pi_nZ=\underline V$ and $\pi_{n+1}Z=\underline W$. Its Postnikov connecting class is the two-extension \eqref{eq:twoextension}.

To check that this class remains nonzero after forgetting the abelian structure, apply \eqref{eq:uct} to $K(\underline V,n)$. Below degree $2n$, its only nonzero reduced rational homology is $\underline V$ in degree $n$. Since $n+2<2n$, the spectral sequence identifies
\begin{equation}\label{eq:ext-two-cohomology}
 \widetilde H^{n+2}_{\C}(K(\underline V,n);\underline W)\cong\Ext^2_{\B}(\underline V,\underline W).
\end{equation}
No differential can enter or leave this term in the indicated total degree. Under this identification the Postnikov invariant of $Z$ is $e$; this also follows directly by truncating its two-term chain complex. A product of the two Eilenberg--MacLane diagrams has zero invariant. Hence $Z$ does not split even as a diagram of spaces, contradicting (iii).
\end{proof}

\begin{remark}\label{rem:stable}
The nonsplit examples in the proof are strict abelian group diagrams and admit arbitrarily many compatible deloopings. Thus objectwise rational H-space splittings, and even objectwise infinite loop space structures, do not by themselves give a splitting of the whole diagram. The obstruction lies in the coefficient category.
\end{remark}

\Needspace{12\baselineskip}

\begin{theorem}\label{thm:equivalence}
Suppose $\C$ is good and satisfies the $\otimes$-condition. The following are equivalent:
\begin{enumerate}
\item $\cdim_{\Q}(\C)\leq1$;
\item every simply connected rational diagram of finite type is algebraic;
\item every simply connected two-stage diagram of simplicial rational vector spaces of finite type is algebraic;
\item every simply connected H-diagram of finite rational type splits as in \eqref{eq:h-splitting}.
\end{enumerate}
\end{theorem}
\begin{proof}
The implication (i)$\Rightarrow$(ii) is \cref{cor:gaps}, and (ii)$\Rightarrow$(iii) is immediate. The equivalence of (i) and (iv) is \cref{thm:h-splitting}.

To prove (iii)$\Rightarrow$(i), suppose $\cdim_{\Q}(\C)>1$. We may choose simple $\underline V,\underline W\in\B$ with $\Ext^2_{\B}(\underline V,\underline W)\ne0$. Otherwise a finite-length induction in each variable would give vanishing for all finite-dimensional modules and hence heredity. Dualize to obtain a nonzero class in $\Ext^2_{\R}(\underline W^{\vee},\underline V^{\vee})$. Since $\underline W^{\vee}$ is simple and the resolution of $\underline V^{\vee}$ is minimal, this class is represented by a nonzero map
\[
 \alpha:\underline W^{\vee}\longrightarrow \Inj^2(\underline V^{\vee}).
\]
There are no coboundaries in the Hom complex from $\underline W^{\vee}$ into that resolution.

Start the geometric model with $\Sym \Inj^\bullet_{n}(\underline V^{\vee})$, $n\geq3$, and attach $\underline W^{\vee}$ in degree $n+1$ using a linear extension of $\alpha$ to its injective resolution. This is the geometric model of the two-stage abelian diagram constructed from the corresponding two-extension. On the socle $\underline W^{\vee}\subset \Inj^0(\underline W^{\vee})$ of the new generators, its linear differential is $\alpha\ne0$; the internal resolution differential vanishes there. This geometric model is not algebraically minimal, contradicting (iii).
\end{proof}

\begin{remark}[Checking the dimension bound by factorization]\label{rem:factorization}
Li's characterization \cite[Theorem~1.2 and Proposition~2.8]{Li11} says that $k\C$ is hereditary precisely when $\C$ is a finite free EI-category and all automorphism-group orders are invertible in $k$. In characteristic zero the latter condition is automatic. The characterization stated over an algebraically closed field also applies to $\Q$: for finite-dimensional algebras with separable semisimple quotient, heredity is preserved and reflected by scalar extension. Here that quotient is $\prod_c\Q G_c$.

For a skeleton of $\C$, the required unique factorization property is concrete. Call a nonisomorphism \emph{unfactorizable} if it is not a composite of two nonisomorphisms. Given two factorizations of the same arrow into unfactorizable arrows, their lengths and intermediate objects must agree. Writing these factorizations as $\alpha_r\cdots\alpha_1$ and $\beta_r\cdots\beta_1$, there must be intermediate automorphisms $g_i$ such that
\[
 \beta_i=g_i\alpha_i g_{i-1}^{-1}\qquad(1\leq i\leq r),
 \qquad g_0=g_r=\id.
\]
Every nonisomorphism admits such a factorization: repeatedly factor any remaining factorizable arrow; finiteness and the EI-condition force this process to stop \cite[Proposition~2.6]{Li11}. The displayed compatibility is Li's unique factorization property \cite[Definition~2.7]{Li11}. By his Proposition~2.8 and Theorem~1.2, it is equivalent over $\Q$ to $\cdim_{\Q}(\C)\leq1$. This gives a finite test: for each arrow, enumerate its factorizations into unfactorizable arrows and check equality of lengths and intermediate objects, then the displayed equations in the finite automorphism groups.

For a finite poset $P$, regarded as a category with one arrow $u\to v$ when $u\leq v$, an unfactorizable arrow is a \emph{cover}: $u<v$ with no element strictly between them. The directed \emph{Hasse diagram} has these covers as edges. A \emph{saturated chain} from $u$ to $v$ is a directed path of covers, hence precisely a factorization of $u\to v$ into unfactorizable arrows. All automorphism groups are trivial, so Li's criterion reduces to
\[
 \cdim_{\Q}(P)\leq1
 \quad\Longleftrightarrow\quad
 \text{each comparable pair has a unique saturated chain.}
\]
The \emph{free category} on a directed graph has paths as morphisms and concatenation as composition. The natural functor from the free category on the Hasse diagram to $P$ identifies all paths with the same endpoints; it is an isomorphism exactly under the displayed uniqueness condition. This explains the phrase ``free on its Hasse diagram.'' For example, a linearly ordered poset satisfies the condition. In the square of \cref{subsec:square}, the paths $4\to2\to1$ and $4\to3\to1$ have the same composite but different intermediate objects. They violate uniqueness, and \eqref{eq:square-resolution} computes the resulting dimension as two.
\end{remark}

\begin{corollary}[Finite group actions]\label{cor:groups}
Let $G$ be finite. Every pointed $G$-space whose fixed-point spaces are simply connected H-spaces with a compatible equivariant multiplication and have finite rational type splits rationally into equivariant Eilenberg--MacLane spaces if and only if $G$ is trivial or cyclic of prime-power order.
\end{corollary}
\begin{proof}
By \cref{prop:orbit}, the orbit category satisfies all the algebraic hypotheses. Elmendorf's correspondence identifies its contravariant diagram homotopy theory with genuine equivariant homotopy theory \cite{Elmendorf83}; it preserves products at the level of the homotopy category. It remains to identify when $\Q\Orb_G$ is hereditary.

If $G$ is cyclic of prime-power order, its subgroups form a chain. Every orbit map factors along that chain, and two such factorizations differ by automorphisms of the intermediate orbits. Thus $\Orb_G$ is a free EI-category and its algebra is hereditary \cite{Li11}. The trivial group is immediate.

Conversely, if the orbit category is free, every two maximal subgroup chains from $1$ to $G$ are conjugate, since they factor the unique map $G/1\to G/G$. In particular, all maximal subgroups are conjugate. A finite group cannot be the union of the conjugates of one proper subgroup: for the transitive action on its cosets, the average number of fixed points is one, while the identity has more than one fixed point, so some group element has no fixed point. If $G$ were noncyclic, every element would generate a proper subgroup and hence belong to a maximal subgroup. Therefore $G$ is cyclic. In a cyclic group the maximal subgroups are individually conjugacy-invariant, so there is only one; its order is a power of a single prime. Now apply \cref{thm:h-splitting}.
\end{proof}

This recovers the finite-group splitting phenomenon studied in \cite{Triantafillou83} from the category-algebra criterion, with a proof of the diagram splitting supplied above.

\clearpage
\section{Examples of minimal models}\label{sec:examples}
We work over $\Q$. Each model is displayed as a diagram of DGCAs. Subscripts on generators give their degrees and are suppressed in differential formulas. Unlisted differentials vanish. Unless stated otherwise, arrows fix the generators common to their source and target and send omitted generators to zero; the arrow labels record these zero images.

\subsection{Spheres and projective spaces}\label{subsec:classical-examples}
When $\C$ has one object and only the identity morphism, every coefficient system is injective and its socle is the whole vector space. Both notions of minimality therefore reduce to Sullivan minimality. The basic models are
\begin{align*}
 \M(S^{2r+1})&=(\Sym(x_{2r+1}),0),&&r\geq1,\\
 \M(S^{2r})&=(\Sym(x_{2r},y_{4r-1}),d),&&dy=x^2,\quad r\geq1,\\
 \M(\mathbb{CP}^{m})&=(\Sym(x_2,y_{2m+1}),d),&&dy=x^{m+1},\quad m\geq1.
\end{align*}
These classical calculations are recalled in \cite[Section~2.1]{Hess06}. In the odd-sphere model the generator is exterior. In the other two models, multiplication by $x^2$ or $x^{m+1}$ is injective in $\Q[x]$, so the cohomology is respectively $\Q[x]/(x^2)$ or $\Q[x]/(x^{m+1})$. The differentials are decomposable, and adjoining $y$ is an elementary extension after adjoining $x$.

Tensor products give the models of products. For example,
\[
 \M(S^2\times S^3)
 =\bigl(\Sym(x_2,u_3,z_3),\ du=x^2,\ dz=0\bigr).
\]
Here $u$ kills the square of the degree-two cohomology class, while $z$ represents the independent degree-three class.

\subsection{Eilenberg--MacLane diagrams over an arrow}\label{ex:arrow-em}
Let $\C=(c\to a)$ and $n\geq2$. For the space diagram $K(\Q,n)\to *$, with $K(\Q,n)$ at $a$, both minimal models are
\begin{equation}\label{eq:example-arrow-em}
\begin{tikzcd}[column sep=large]
 \begin{gathered}c:\ \Sym(u_n,v_{n+1})\\du=v\end{gathered}
 \arrow[r,"v\mapsto0"]
 &\begin{gathered}a:\ \Sym(u_n)\\d=0\end{gathered}.
\end{tikzcd}
\end{equation}
The generators form the minimal injective resolution
\begin{equation}\label{eq:example-arrow-resolution}
 0\longrightarrow(0\to\Q)
 \longrightarrow(\Q\xrightarrow{\id}\Q)
 \longrightarrow(\Q\to0)\longrightarrow0,
\end{equation}
placed in degrees $n,n+1$. This proves geometric minimality. The socle generators are $u$ at $a$ and $v$ at $c$, and their differentials vanish there, proving algebraic minimality. The left algebra is contractible: extending $h(v)=u$, $h(u)=0$ as a derivation makes $dh+hd$ the word-length operator, invertible on the augmentation ideal. The right algebra models $K(\Q,n)$. Although $(u,v)$ is contractible at $c$, it cannot be cancelled naturally because $u$ has a nonzero restriction to $a$.

For the reverse space diagram $*\to K(\Q,n)$, both minimal models are
\[
\begin{tikzcd}[column sep=large]
 \begin{gathered}c:\ \Sym(u_n)\\d=0\end{gathered}
 \arrow[r,"u\mapsto0"]&a:\ \Q.
\end{tikzcd}
\]
Here the generator system $(\Q\to0)$ is already injective.

For a linear map $\phi:U\to W$ between finite-dimensional spaces, choose splittings of its dual and put
\[
 K=\ker\phi^\vee,\qquad R=\im\phi^\vee,\qquad L=\coker\phi^\vee.
\]
Thus $W^\vee\cong K\oplus R$ and $U^\vee\cong R\oplus L$. A subscript places a vector space in that degree, and $L'$ is a second copy of $L$. Both minimal models of $K(U,n)\to K(W,n)$ are the explicit DGCA diagram
\begin{equation}\label{eq:example-arrow-general}
\begin{tikzcd}[row sep=large]
 \begin{gathered}
 c:\ \Sym(K_n\oplus R_n\oplus L_n\oplus L'_{n+1})\\
 d\ell=\ell'\quad(\ell\in L)
 \end{gathered}\arrow[d,"{K,L'\mapsto0}"]\\
 \begin{gathered}a:\ \Sym(R_n\oplus L_n)\\d=0\end{gathered}.
\end{tikzcd}
\end{equation}
The arrow fixes $R$ and $L$. Each pair $(\ell,\ell')$ is a copy of \eqref{eq:example-arrow-em}; the remaining generators are closed injective systems. This proves both minimality conditions. Cancelling these pairs at $c$ recovers the map $\phi^\vee$. Different splittings give isomorphic minimal diagrams.

\subsection{Tubular-neighbourhood squares}\label{ex:tubular-square}
Let $M\hookrightarrow N$ be a smooth embedding of a closed submanifold of positive codimension, with normal bundle $\nu$. Choose a closed tubular neighbourhood $T=D(\nu)$, and put
\[
 B=\partial T=S(\nu),\qquad
 E=\overline{N\setminus T}=N\setminus\operatorname{int}T.
\]
The inclusions form the pushout square
\begin{equation}\label{eq:tubular-topological}
\begin{tikzcd}[column sep=large,row sep=large]
 B\arrow[r,hook]\arrow[d,hook]&T\arrow[d,hook]\\
 E\arrow[r,hook]&N=T\cup_B E .
\end{tikzcd}
\end{equation}
Both inclusions out of $B$ are cofibrations, by collars, so this is also a homotopy pushout. Assume that its four spaces are simply connected and have finite rational type. A sufficient geometric hypothesis is that $M,N$ are closed simply connected manifolds and $\operatorname{codim}(M,N)\geq3$: the sphere bundle $B\to M$ has simply connected fibre, and general position shows that $E\simeq N\setminus M$ is simply connected. A point of $B$ supplies compatible basepoints.

\smallskip\noindent\textbf{A general construction.}
Use the square category of \cref{subsec:square}, with $4=N$, $2=T$, $3=E$, and $1=B$. Write $\M_N,\M_T,\M_E,\M_B$ for the four algebras of a model.
Choose an augmented DGCA model of the boundary span, written in the algebra direction as
\[
 R\xrightarrow{f} A\xleftarrow{g} S,
 \qquad R\simeq\A_{\PL}(T),\quad
 A\simeq\A_{\PL}(B),\quad S\simeq\A_{\PL}(E).
\]
Here the two maps are part of the model data; the individual algebras alone do not specify the gluing. Replace $f$, if necessary, by a surjective map $\widetilde f:\widetilde R\twoheadrightarrow A$ through a quasi-isomorphism $R\qi\widetilde R$. One explicit choice uses the augmented path algebra
\[
 A^I=\Q\oplus\bigl(A_+\otimes\Q[t,dt]\bigr),\qquad
 \widetilde R=R\times_{A,\mathrm{ev}_0}A^I,
 \qquad \widetilde f=\mathrm{ev}_1.
\]
Constant paths give the quasi-isomorphism; the kernel of evaluation at $0$ is contracted by integration. Evaluation at $1$ is surjective, since $(0,ta)$ lifts each $a\in A_+$. Consequently the strict pullback square
\begin{equation}\label{eq:tubular-pullback}
\begin{tikzcd}[column sep=large,row sep=large]
 P=\widetilde R\times_A S\arrow[r]\arrow[d]
 &\widetilde R\arrow[d,two heads,"\widetilde f"]\\
 S\arrow[r,"g"']&A
\end{tikzcd}
\end{equation}
is a homotopy pullback of DGCAs and models the whole square \eqref{eq:tubular-topological}, with $P$ modelling $N$. Indeed, polynomial forms turn the homotopy pushout into a homotopy pullback. This can also be checked on a compatible triangulation: restriction of polynomial forms to a subcomplex is surjective, and the forms on the union are the pullback of the forms on the two pieces. Models of such Poincar\'e embedding squares are studied in \cite{LambrechtsStanley05}; the construction here uses the full boundary span and imposes no additional high-codimension hypothesis once the four spaces satisfy our connectivity and finiteness assumptions.

Apply the geometric construction to \eqref{eq:tubular-pullback}, using \eqref{eq:tubular-envelope} to resolve each homotopy-dual system. The resolutions have length at most two, so a degree-$n$ homotopy system contributes generators only in degrees $n,n+1,n+2$. Group the generators according to the vertices supporting their socles: $W_1$ occurs at every vertex, $W_2$ at $N,T$, $W_3$ at $N,E$, and $W_4$ only at $N$. The minimal model is a DGCA diagram of the form
\begin{equation}\label{eq:tubular-general-model}
\begin{tikzcd}[column sep={3.2cm,between origins},row sep=large]
 &N:\ (\Sym(W_1\oplus W_2\oplus W_3\oplus W_4),d)
 \arrow[dl,"{W_3,W_4\mapsto0}"']
 \arrow[dr,"{W_2,W_4\mapsto0}"]&\\
 T:\ (\Sym(W_1\oplus W_2),d)\arrow[dr,"W_2\mapsto0"']
 &&E:\ (\Sym(W_1\oplus W_3),d)\arrow[dl,"W_3\mapsto0"]\\
 &B:\ (\Sym W_1,d).&
\end{tikzcd}
\end{equation}
Here $d$ is the elementary-extension differential \eqref{eq:extension}, consisting of the resolution terms and the extended $k$-invariant cocycles; at the lower vertices it is obtained by the displayed projections. The boundary span determines these cocycles. Natural cancellation by \cref{lem:cancellation,thm:algebraic-existence} gives the algebraically minimal model in the same diagram form, with fewer generators when cancellation occurs. One must transport the full differential, including its nonlinear terms.

\smallskip\noindent\textbf{The standard sphere embedding.}
Let $p,q\geq3$ be odd, and take the standard embedding $S^p\subset S^{p+q+1}$. Its normal bundle is trivial, and the decomposition
\[
 \partial(D^{p+1}\times D^{q+1})
 =(S^p\times D^{q+1})\cup_{S^p\times S^q}(D^{p+1}\times S^q)
\]
identifies, after smoothing corners, the tubular square with
\[
 B=S^p\times S^q,\qquad
 T=S^p\times D^{q+1},\qquad E=D^{p+1}\times S^q,\qquad
 N=S^{p+q+1}.
\]
A model minimal in both diagram senses is
\begin{equation}\label{eq:tubular-sphere-model}
\begin{tikzcd}[column sep={3.2cm,between origins},row sep=large]
 &\begin{gathered}
 N:\ \Sym(x_p,y_q,a_{p+1},b_{q+1},z_{p+q+1})\\
 dx=a,\quad dy=b,\quad dz=ab
 \end{gathered}\arrow[dl,"{a,z\mapsto0}"']
 \arrow[dr,"{b,z\mapsto0}"]&\\
 \begin{gathered}T:\ \Sym(x_p,y_q,b_{q+1})\\dy=b\end{gathered}
 \arrow[dr,"b\mapsto0"']&&
 \begin{gathered}E:\ \Sym(x_p,a_{p+1},y_q)\\dx=a\end{gathered}
 \arrow[dl,"a\mapsto0"]\\
 &\begin{gathered}B:\ \Sym(x_p,y_q)\\d=0\end{gathered}.&
\end{tikzcd}
\end{equation}
Unlisted differentials vanish. Every arrow preserves the generators present at its target and kills the others. In particular the two composites agree, and $dz=ab$ restricts to zero on both sides. At $N$ the substitution
\[
 \zeta=z-xb,\qquad d\zeta=0,
\]
gives the objectwise decomposition
\[
 \M_N\cong\Sym(\zeta_{p+q+1})
 \otimes(\Sym(x_p,a_{p+1}),\ dx=a)
 \otimes(\Sym(y_q,b_{q+1}),\ dy=b).
\]
The last two factors are contractible; similarly $\M_T\simeq\Sym(x_p)$ and $\M_E\simeq\Sym(y_q)$. This verifies the four objectwise homotopy types, but the gluing requires a further check.

The restriction maps $r_T:\M_T\to\M_B$ and $r_E:\M_E\to\M_B$ are surjective. Put $L=\M_T\times_{\M_B}\M_E$ and let $\rho:\M_N\to L$ be the map induced by the two restrictions. The short exact sequence of complexes
\[
 0\longrightarrow L\longrightarrow\M_T\oplus\M_E
 \xrightarrow{\,r_T-r_E\,}\M_B\longrightarrow0
\]
gives $H^*(L)\cong\Sym(\eta_{p+q+1})$, where $\eta=\delta[xy]$. Indeed, the maps on cohomology identify the classes $x$ and $y$ separately, and the only remaining positive-degree class of $\M_B$ is $xy$. Its connecting class is represented by $(-xb,0)$, since $d(xy)=-xb$ at $T$. But
\[
 \rho(\zeta)=(-xb,0),
\]
so $\rho$ is a quasi-isomorphism. The inclusions of the closed generators $x$ at $T$ and $y$ at $E$, together with the identity at $B$, identify the remaining span with the usual models of the two projections $S^p\times S^q\to S^p,S^q$. Therefore \eqref{eq:tubular-sphere-model} models the complete tubular square, including its homotopy-pushout structure.

For geometric minimality, $(x,a)$ and $(y,b)$ are the two-term minimal injective resolutions of the degree-$p$ and degree-$q$ homotopy-dual systems. Their kernels are supported on $T,B$ and $E,B$, respectively. Adjoin them in increasing homotopy degree, together when $p=q$, and then adjoin $z$ in degree $p+q+1$. Its system is injective and supported only at $N$, and the cocycle $ab$ restricts to zero on both sides. These are elementary extensions. The socle generators are $x,y$ at $B$, $a$ at $E$, $b$ at $T$, and $z$ at $N$; their differentials are zero or decomposable. Hence the displayed model is also algebraically minimal.

In particular, setting $p=q=3$ in the displayed diagram gives the tubular square of $S^3\subset S^7$.
The nonlinear term $ab$ records the gluing. Replacing $dz=ab$ by $dz=0$ preserves all four objectwise rational homotopy types, but the ambient sphere class then maps to zero in $H^{p+q+1}(L)$, so the square ceases to be a homotopy pullback of DGCAs.

\begin{example}[Homotopy automorphisms of the tubular square]\label{ex:tubular-automorphisms}
Let $\X$ be the tubular square of the standard embedding $S^3\subset S^7$, with all four vertices labelled. Its group of pointed rational homotopy automorphisms is
\begin{equation}\label{eq:tubular-automorphism-group}
 \mathcal E_{\Q}(\X)\cong\Q^\times\times\Q^\times.
\end{equation}
Under this identification, $(\lambda,\mu)$ acts by $\lambda$ on $H^3(T;\Q)$, by $\mu$ on $H^3(E;\Q)$, and by $\lambda\mu$ on $H^7(N;\Q)$.

Indeed, take $p=q=3$ in \eqref{eq:tubular-sphere-model}. Since the structure maps are surjective, a natural automorphism is determined at $N$ and must preserve the two restriction kernels $(a,z)$ and $(b,z)$. Degree considerations give $\phi(x)=\lambda x+\beta y$ and $\phi(y)=\gamma x+\mu y$. The conditions $d\phi(x)\in(a,z)$ and $d\phi(y)\in(b,z)$ force $\beta=\gamma=0$. Moreover,
\[
 \M_N^7=\langle z,xa,xb,ya,yb\rangle_{\Q},
 \qquad\bigl((a,z)\cap(b,z)\bigr)^7=\Q z,
\]
so $\phi(z)=\nu z$. Commuting with $d$ gives $\nu=\lambda\mu$. Conversely, every $\lambda,\mu\in\Q^\times$ defines a natural DGCA automorphism by
\[
 (x,y,a,b,z)\longmapsto
 (\lambda x,\mu y,\lambda a,\mu b,\lambda\mu z).
\]
Its actions on $H^3(T;\Q)$ and $H^3(E;\Q)$ detect $\lambda$ and $\mu$, so it is homotopic to the identity only when $\lambda=\mu=1$. Thus $\Aut_0(\M)$ is trivial, and \cref{lem:homotopy-automorphisms} proves \eqref{eq:tubular-automorphism-group}; the opposite group is immaterial here because the group is abelian. Finally, the ambient sphere class $[z-xb]$ is multiplied by $\lambda\mu$. Consequently, the subgroup fixing the ambient fundamental cohomology class is
\[
 \{(\lambda,\mu)\in(\Q^\times)^2\mid\lambda\mu=1\}
 \cong\Q^\times.
\]
The nonlinear differential $dz=ab$ therefore forces the degree on the ambient sphere to be the product of the degrees on the two pieces.
\end{example}

\begin{remark}[Extension of boundary self-equivalences]\label{rem:boundary-extensions}
For the standard embedding $S^p\subset S^{p+q+1}$ with odd $p,q\geq3$, the boundary is $B=S^p\times S^q$. Its model $(\Sym(x_p,y_q),0)$ identifies its pointed rational homotopy automorphism group as
\[
 \mathcal E_{\Q}(B)\cong
 \begin{cases}
  (\Q^\times)^2,&p\ne q,\\
  \mathrm{GL}_2(\Q),&p=q.
 \end{cases}
\]
Indeed, a graded automorphism scales the two generators when their degrees differ and makes an invertible linear substitution when their degrees agree; its action on cohomology determines its homotopy class.

Restriction of a self-equivalence of the labelled tubular square preserves each of the two lines in $H^*(B;\Q)$ obtained from $T$ and $E$, namely $\Q\cdot[x]$ and $\Q\cdot[y]$ in their respective generator degrees. Thus its image in $\mathcal E_{\Q}(B)$ consists exactly of the diagonal substitutions $x\mapsto\lambda x$, $y\mapsto\mu y$. Every such substitution extends to \eqref{eq:tubular-sphere-model} by $a\mapsto\lambda a$, $b\mapsto\mu b$, $z\mapsto\lambda\mu z$.

Consequently, when $p\ne q$, every pointed rational self-equivalence of $S^p\times S^q$ is homotopic to the boundary restriction of a self-equivalence of the entire embedding square. When $p=q$, this holds precisely for those preserving the two factor lines. In particular, for $S^3\times S^3$ the boundary automorphism $x\mapsto x+y$, $y\mapsto y$ does not extend, even up to rational homotopy, because it does not preserve $\Q\cdot[x]\subset H^3(B;\Q)$. The group computed in \cref{ex:tubular-automorphisms} is therefore the subgroup of boundary self-equivalences compatible with the labelled embedding data.
\end{remark}

\subsection{A projective line in projective three-space}\label{ex:projective-square}
Consider the standard linear embedding $\mathbb{CP}^1\subset N=\mathbb{CP}^3$, with complex normal bundle $\mathcal O(1)\oplus\mathcal O(1)$. Write $T,B,E$ for its tubular neighbourhood, boundary, and exterior as in \cref{ex:tubular-square}. In homogeneous coordinates $[s,t]\in\mathbb P(\mathbb C^2\oplus\mathbb C^2)$, take $T$ and $E$ to be the regions $\|s\|\geq\|t\|$ and $\|s\|\leq\|t\|$. Each retracts onto a projective line, and
\[
 B\cong(S^3\times S^3)/S^1,
\]
where the circle acts diagonally. The boundary maps to the two retracts are $[s,t]\mapsto[s]$ and $[s,t]\mapsto[t]$.

The complete tubular square has the following model, minimal in both senses:
\begin{equation}\label{eq:projective-square-model}
\begin{tikzcd}[column sep={3.2cm,between origins},row sep=large]
 &\begin{gathered}
 N:\ \Sym(x_2,u_3,v_3,a_4,b_4,z_7)\\
 du=x^2+a,\quad dv=x^2+b,\quad dz=ab
 \end{gathered}\arrow[dl,"{a,z\mapsto0}"']
 \arrow[dr,"{b,z\mapsto0}"]&\\
 \begin{gathered}T:\ \Sym(x_2,u_3,v_3,b_4)\\du=x^2,\quad dv=x^2+b\end{gathered}
 \arrow[dr,"b\mapsto0"']&&
 \begin{gathered}E:\ \Sym(x_2,u_3,v_3,a_4)\\du=x^2+a,\quad dv=x^2\end{gathered}
 \arrow[dl,"a\mapsto0"]\\
 &\begin{gathered}B:\ \Sym(x_2,u_3,v_3)\\du=dv=x^2\end{gathered}.&
\end{tikzcd}
\end{equation}
All four arrows are quotient maps, and $d^2=0$ on every generator.

To identify the boundary span, view $B$ as the homotopy pullback of the two classifying maps $\mathbb{CP}^1\to BS^1$ of the Hopf bundles. Indeed, the fibre product over $BS^1$ of their Borel fibrations is the diagonal Borel construction of $S^3\times S^3$, equivalent to $B$. Tensoring their relative Sullivan models over $(\Sym(x_2),0)$ gives the span
\[
 (\Sym(x_2,u_3),\ du=x^2)
 \longrightarrow\M_B\longleftarrow
 (\Sym(x_2,v_3),\ dv=x^2),
\]
with the evident inclusions; see \cite[Theorem~2.4]{Hess06}. These inclusions factor through $\M_T$ and $\M_E$, respectively, by quasi-isomorphisms: at $T$ the pair $(v-u,b)$ is contractible, and at $E$ the pair $(u-v,a)$ is contractible. At $B$, the closed generator $v-u$ splits off, giving $B\simeq_{\Q}S^2\times S^3$.

At $N$, make the invertible change of generators
\[
 a'=a+x^2,\qquad b'=b+x^2,\qquad w=z-ub+x^2v.
\]
Then $du=a'$, $dv=b'$, and
\[
 dw=ab-(x^2+a)b+x^2(x^2+b)=x^4.
\]
Cancelling $(u,a')$ and $(v,b')$ leaves $(\Sym(x_2,w_7),dw=x^4)$, the usual model of $\mathbb{CP}^3$.

For the gluing, put $L=\M_T\times_{\M_B}\M_E$. The restrictions are surjective, so $L$ computes their homotopy pullback. The exact sequence for the difference of the two restrictions gives one-dimensional cohomology in degrees $0,2,4,6$ and zero otherwise. If $X=[(x,x)]\in H^2(L)$ and $\delta$ is the connecting map, then
\[
 X^2=\delta[u-v]\ne0,\qquad
 X^3=\delta[x(u-v)]\ne0.
\]
Indeed, $(u,v)$ lifts $u-v$ and has differential $(x^2,x^2)$; multiplying by $x$ proves the second identity. Thus $H^*(L)=\Q[X]/(X^4)$. The natural map $\M_N\to L$ sends $[x]$ to $X$ and is a quasi-isomorphism. Together with the identified boundary span, this proves that \eqref{eq:projective-square-model} models the whole tubular square.

Finally, $x$ is a closed injective generator system in degree two. The pairs $(u,a)$ and $(v,b)$ resolve the degree-three homotopy-dual systems supported on $T,B$ and $E,B$, with attachment terms $x^2$. Adjoin these minimal resolutions after $x$, and then adjoin the injective generator $z$ supported at $N$, with $dz=ab$. This is the geometric minimal filtration. The linear differential is $d_0u=a$, $d_0v=b$; it vanishes on all socle generators, namely $x,u,v$ at $B$, $a$ at $E$, $b$ at $T$, and $z$ at $N$. Hence the model is also algebraically minimal.

\begin{example}[Automorphisms of the projective embedding and its boundary]\label{ex:projective-automorphisms}
Let $\X$ be this tubular square, with all four vertices labelled. Its pointed rational homotopy automorphism group is
\begin{equation}\label{eq:projective-automorphism-group}
 \mathcal E_{\Q}(\X)\cong\Q^\times.
\end{equation}
The element $\lambda\in\Q^\times$ is represented contravariantly on \eqref{eq:projective-square-model} by
\[
 x\longmapsto\lambda x,\qquad
 (u,v,a,b)\longmapsto\lambda^2(u,v,a,b),\qquad
 z\longmapsto\lambda^4z.
\]
Indeed, a natural DGCA automorphism is determined at $N$ and preserves the restriction ideals $(a,z)$ and $(b,z)$. Necessarily $\phi(x)=\lambda x$. If $\phi(u)=r u+s v$, then
\[
 \phi(a)=d\phi(u)-\phi(x)^2
 =(r+s-\lambda^2)x^2+ra+sb\in(a,z),
\]
so $r=\lambda^2$ and $s=0$. Similarly $\phi(v)=\lambda^2v$, and hence $\phi(a)=\lambda^2a$, $\phi(b)=\lambda^2b$. Since
\[
 \bigl((a,z)\cap(b,z)\bigr)^7=(ab,z)^7=\Q z,
\]
we have $\phi(z)=\nu z$, and $d\phi(z)=\phi(ab)$ gives $\nu=\lambda^4$. Conversely, the displayed substitutions define natural automorphisms. Distinct $\lambda$ act differently on $H^2(N;\Q)$, so none are identified by homotopy. Now apply \cref{lem:homotopy-automorphisms}.

For the boundary, put $y=v-u$. Its ordinary minimal Sullivan model is
\[
 (\Sym(x_2,u_3,y_3),d),\qquad du=x^2,\quad dx=dy=0.
\]
Every pointed rational self-equivalence is represented uniquely by a substitution
\begin{equation}\label{eq:projective-boundary-automorphisms}
 f_{c,\lambda,\mu}^{*}(x)=\lambda x,\qquad
 f_{c,\lambda,\mu}^{*}(y)=\mu y,\qquad
 f_{c,\lambda,\mu}^{*}(u)=\lambda^2u+\mu c y,
\end{equation}
where $c\in\Q$ and $\lambda,\mu\in\Q^\times$. These are all possibilities by degree and the differential equation. All three parameters are detected by the induced maps on the indecomposables, dual to $\pi_2(B)\otimes\Q$ and $\pi_3(B)\otimes\Q$, so distinct substitutions are not homotopic \cite[Theorems~1.24--1.25]{Hess06}. Taking account of the reversal of composition by pullback gives
\[
 (c,\lambda,\mu)(c',\lambda',\mu')
 =\left(c+\frac{\lambda^2}{\mu}c',\lambda\lambda',\mu\mu'\right).
\]
Consequently,
\begin{equation}\label{eq:projective-boundary-group}
 \mathcal E_{\Q}(B)\cong(\Q,+)\rtimes(\Q^\times)^2,
 \qquad (\lambda,\mu)\cdot c=\frac{\lambda^2}{\mu}c.
\end{equation}

Restriction to the boundary is the injective homomorphism
\begin{equation}\label{eq:projective-boundary-restriction}
 \mathcal E_{\Q}(\X)\longrightarrow\mathcal E_{\Q}(B),
 \qquad\lambda\longmapsto(0,\lambda,\lambda^2).
\end{equation}
Its image is exactly the centre of \eqref{eq:projective-boundary-group}: a central element must have $c=0$ and $\lambda^2/\mu=1$. Thus a boundary self-equivalence extends over the labelled square precisely when it has no shear and its multiplier on $H^3(B;\Q)$ is the square of its multiplier on $H^2(B;\Q)$. Such an extension is unique in the rational diagram homotopy category.

The shear is essential information beyond cohomology. For example, $x\mapsto x$, $y\mapsto y$, $u\mapsto u+y$ acts identically on
\[
 H^*(B;\Q)=\Q[x]/(x^2)\otimes\Sym(y_3),
\]
but adds the rational Hopf class of the $S^2$ factor to the $S^3$ generator in $\pi_3(S^2\times S^3)\otimes\Q$. It therefore cannot extend over the square, even up to homotopy. The two boundary projections, represented by the generators $u$ and $v=u+y$, detect this obstruction.

The extension criterion above allows the ambient component to be any rational self-equivalence. The ordinary model $(\Sym(x_2,w_7),dw=x^4)$ shows that evaluation at $N$ identifies \eqref{eq:projective-automorphism-group} with $\mathcal E_{\Q}(\mathbb{CP}^3)$. Thus each ambient rational self-equivalence extends uniquely to the labelled square, up to diagram homotopy. Its action on the ambient fundamental cohomology class $[x^3]$ is multiplication by $\lambda^3$, as is its action on the boundary class $[xy]\in H^5(B;\Q)$. Every nonzero rational cube is an allowed ambient degree. For example, the nonidentity boundary automorphism $(c,\lambda,\mu)=(0,2,4)$ extends to the square; its ambient component has degree $8$.

Requiring the extension to fix the ambient fundamental cohomology class is an \emph{additional condition}. It imposes $\lambda^3=1$, hence $\lambda=1$ over $\Q$. Under this additional requirement, a boundary self-equivalence extends if and only if it is homotopic to the identity, and its extension is the identity diagram homotopy class. With ambient degree unrestricted, the entire family $(0,\lambda,\lambda^2)$, $\lambda\in\Q^\times$, extends.
\end{example}

\subsection{A reflection action on an odd sphere}\label{ex:reflection}
Let $G=C_2=\{1,\sigma\}$ act on $S^{2r+1}\subset\mathbb R^{2r+2}$ by reflection in one coordinate, with $r\geq1$. Choose a basepoint in the fixed sphere $S^{2r}$. The orbit category has objects $c=G/1$ and $a=G/G$, a unique arrow $c\to a$, and automorphism group $C_2$ at $c$. The fixed-point diagram is
\[
 \X(a)=S^{2r}\longrightarrow\X(c)=S^{2r+1}.
\]
Reflection has degree $-1$ on the ambient sphere. An ordinary DGCA diagram for these fixed-point data is therefore
\begin{equation}\label{eq:example-reflection-ordinary}
\begin{tikzcd}[column sep=large]
 \begin{gathered}c:\ \Sym(x_{2r+1})\\d=0\end{gathered}
 \arrow[loop left,"\sigma"]\arrow[r,"x\mapsto0"]
 &\begin{gathered}a:\ \Sym(u_{2r},z_{4r-1})\\dz=u^2\end{gathered},
\end{tikzcd}
\end{equation}
with $\sigma(x)=-x$. To realize this strictly on polynomial forms, choose an anti-invariant cocycle for the fundamental class by applying $(1-\sigma^*)/2$ to any representative. Its restriction to the fixed sphere is zero. Ordinary cocycle choices at $a$ then give a natural objectwise quasi-isomorphism from \eqref{eq:example-reflection-ordinary} to $\A_{\PL}(\X)$.

Both diagram-minimal models are given by
\begin{equation}\label{eq:example-reflection-model}
\begin{tikzcd}[row sep=large]
 \begin{gathered}
 c:\ \Sym(u_{2r},q_{2r+1},x_{2r+1},z_{4r-1},w_{4r})\\
 du=q,\quad dz=u^2-w,\quad dw=2uq
 \end{gathered}\arrow[loop left,"\sigma"]
 \arrow[d,"{q,x,w\mapsto0}"]\\
 \begin{gathered}a:\ \Sym(u_{2r},z_{4r-1})\\dz=u^2\end{gathered}.
\end{tikzcd}
\end{equation}
The restriction preserves $u,z$ and kills $q,x,w$. At $c$, the action fixes $u,q,z,w$ and sends $x$ to $-x$. Naturality is immediate. Since $q$ is odd, $d^2w=2q^2=0$, and
\[
 d^2z=d(u^2-w)=2uq-2uq=0.
\]
Thus the correction $dw=2uq$ is required by the extended-cocycle equation; it cannot be omitted.

At $c$, write $w'=w-u^2$. Then $dw'=0$ and $dz=-w'$, so
\[
 \M(c)\cong(\Sym(x_{2r+1}),0)
 \otimes(\Sym(u_{2r},q_{2r+1}),\ du=q)
 \otimes(\Sym(z_{4r-1},w'_{4r}),\ dz=-w').
\]
The last two factors are contractible. Sending the ambient generator in \eqref{eq:example-reflection-ordinary} to $x$ and taking the identity at $a$ consequently gives a natural objectwise quasi-isomorphism to \eqref{eq:example-reflection-model}.

For minimality, $u,z$ generate copies of the injective system $(\Q\xrightarrow{\id}\Q)$ with trivial $C_2$-action, while $q,w$ generate injective trivial systems supported at $c$. The generator $x$ spans the injective sign system supported at $c$. The pairs $(u,q)$ and $(z,w)$ arise from the minimal resolutions of the fixed-point homotopy systems in degrees $2r$ and $4r-1$; $x$ is adjoined as a closed generator in degree $2r+1$. When $r=1$, $x$ and $z$ are introduced together at the degree-three stage. This gives the geometric filtration. On the socle, the only potentially nonzero differential at $c$ is $dw=2uq$, which is decomposable, and the differential at $a$ is also decomposable. Hence the model is algebraically minimal. This example exhibits both the automorphism action and the additional generators needed for naturality over the orbit category.

\Needspace{8\baselineskip}
\subsection{A family with different minimal models}\label{ex:square-family}
Replacing degrees $3,4$ in \cref{subsec:square} by $n,n+1$, for $n\geq3$, gives the algebraically minimal DGCA square
\begin{equation}\label{eq:example-square-family}
\begin{tikzcd}[column sep={3.2cm,between origins},row sep=large]
 &\begin{gathered}4:\ \Sym(x_n,y_{n+1},z_{n+1})\\dx=y+z\end{gathered}
 \arrow[dl,"z\mapsto0"']\arrow[dr,"y\mapsto0"]&\\
 \begin{gathered}2:\ \Sym(x_n,y_{n+1})\\dx=y\end{gathered}
 \arrow[dr,"y\mapsto0"']&&
 \begin{gathered}3:\ \Sym(x_n,z_{n+1})\\dx=z\end{gathered}
 \arrow[dl,"z\mapsto0"]\\
 &\begin{gathered}1:\ \Sym(x_n)\\d=0\end{gathered}.&
\end{tikzcd}
\end{equation}
The generator complex consists of $\Ic{1}$ in degree $n$ mapping diagonally to $\Ic{2}\oplus\Ic{3}$ in degree $n+1$. Freeness on this complex gives cofibrancy, and its socle differential is zero. Thus the displayed model is algebraically minimal. Its only homotopy-dual systems are
\[
 H^n\Qind\M_{\mathrm{alg}}=S_1,\qquad
 H^{n+1}\Qind\M_{\mathrm{alg}}=S_4.
\]

The geometrically minimal model is
\begin{equation}\label{eq:example-square-family-geometric}
\begin{tikzcd}[column sep={3.2cm,between origins},row sep=large]
 &\begin{gathered}
 4:\ \Sym(x_n,y_{n+1},z_{n+1},v_{n+1},w_{n+2})\\
 dx=y+z,\quad dy=dv=w,\quad dz=-w
 \end{gathered}\arrow[dl,"{z,v,w\mapsto0}"']
 \arrow[dr,"{y,v,w\mapsto0}"]&\\
 \begin{gathered}2:\ \Sym(x_n,y_{n+1})\\dx=y\end{gathered}
 \arrow[dr,"y\mapsto0"']&&
 \begin{gathered}3:\ \Sym(x_n,z_{n+1})\\dx=z\end{gathered}
 \arrow[dl,"z\mapsto0"]\\
 &\begin{gathered}1:\ \Sym(x_n)\\d=0\end{gathered}.&
\end{tikzcd}
\end{equation}
Its first stage is the full minimal resolution \eqref{eq:square-resolution}, with $w$ in degree $n+2$; its next stage adjoins $v$ in degree $n+1$. Since $v$ belongs to the socle and $dv=w$, this model is not algebraically minimal. The natural substitutions $Y=y-v$, $Z=z+v$ give
\begin{equation}\label{eq:example-square-family-cancellation}
 \M_{\mathrm{geom}}\cong\M_{\mathrm{alg}}
 \otimes\Sym(v_{n+1},w_{n+2}),\qquad dv=w,
\end{equation}
where the additional generator systems are supported at $4$. This removes a contractible pair joining two Postnikov stages.

The example also makes the failure of H-diagram splitting explicit. The differentials and structure maps are linear on generators, so realization gives a simplicial rational vector-space diagram: maps to polynomial forms can be added on generators. Its objectwise homotopy types are $K(\Q,n)$ at $1$, points at $2,3$, and $K(\Q,n+1)$ at $4$. Nevertheless its diagram Postnikov invariant is the nonzero class
\[
 \Ext^2_{\Vect_{\Q}^{\C}}(S_4,S_1)\cong\Q
\]
represented by \eqref{eq:square-resolution}. By the identification \eqref{eq:ext-two-cohomology}, after dualizing the coefficient systems this is also a nonzero invariant in diagram cohomology. Thus the H-diagram does not split into the product of its two Eilenberg--MacLane diagrams. The same two-fold extension is visible algebraically as the pair cancelled in \eqref{eq:example-square-family-cancellation}.

\section*{Acknowledgements}
The author thanks Igor Baskov for discussions of tensor products of injective representations and Nicolai Reshetikhin for his interest and assistance. The author also thanks Urs Schreiber for continuing discussions.

\end{document}